\documentclass[reqno]{amsart}
\usepackage{amsfonts,amsthm, amscd, epsfig, amsmath,
  amssymb,enumerate}

\numberwithin{equation}{section}
\makeatletter
\@addtoreset{equation}{section}
\makeatother

\renewcommand\thefigure{\thesection.\@arabic\c@figure}
\renewcommand\thetable{\thesection.\@arabic\c@table}

\newtheorem{theorem}{Theorem}[section]
\newtheorem{lemma}[theorem]{Lemma}
\newtheorem{proposition}[theorem]{Proposition}
\newtheorem{corollary}[theorem]{Corollary}

\newtheorem{theo}[theorem]{Theorem}

\newtheorem{prop}[theorem]{Proposition}

\newcommand{\RR}{{\mathbb R}}
\newcommand{\ZZ}{{\mathbb Z}}
\newcommand{\LL}{{\mathbb L}}

\newcommand{\PP}{{\mathbb P}}
\newcommand{\EE}{{\mathbb E}}

\newcommand{\Aa}{{\mathcal A}}

\newcommand{\Bb}{{\mathcal B}}

\newcommand{\Ff}{{\mathcal F}}

\newcommand{\Ll}{{\mathcal L}}

\newcommand{\mc}[1]{{\mathcal #1}}
\newcommand{\mf}[1]{{\mathfrak #1}}
\newcommand{\mb}[1]{{\mathbf #1}}
\newcommand{\bb}[1]{{\mathbb #1}}

\let\a=\alpha    \let\d=\delta  \let\e=\varepsilon
 \let\g=\gamma     
\let\l=\lambda
   
\let\o=\omega
\let\p=\pi  

\let\s=\sigma  
\let\t=\tau

\renewcommand{\theenumi}{(\roman{enumi})}
\renewcommand{\labelenumi}{\theenumi}

\newcommand{\<}{\langle}
\renewcommand{\>}{\rangle}

\title[Exclusion process with random conductances]{Hydrodynamic
  behavior of one dimensional subdiffusive exclusion processes with
  random conductances}

\author{A. Faggionato} \address{Dipartimento di Matematica
  ``G. Castelnuovo", Universit\`a ``La Sapienza'', P.le Aldo Moro 2,
  00185 Roma, Italy. \newline
  e-mail: \rm \texttt{faggiona@mat.uniroma1.it}} 

\author{M. Jara}
\address{\noindent IMPA, Estrada Dona Castorina 110, CEP 22460 Rio de
  Janeiro, Brasil \newline e-mail: \rm \texttt{monets@impa.br} }

\author{C. Landim} \address{\noindent IMPA, Estrada Dona Castorina
  110, CEP 22460 Rio de Janeiro, Brasil and CNRS UMR 6085,
  Universit\'e de Rouen, UMR 6085, Avenue de l'Universit\'e, BP.12,
  Technop\^ole du Madrillet, F76801 Saint-\'Etienne-du-Rouvray,
  France.  \newline e-mail: \rm \texttt{landim@impa.br} }

\begin{document}

\begin{abstract}
  Consider a system of particles performing nearest neighbor random
  walks on the lattice $\ZZ$ under hard--core interaction. The rate
  for a jump over a given bond is direction--independent and the
  inverse of the jump rates are i.i.d. random variables belonging to
  the domain of attraction of an $\a$--stable law, $0<\a<1$. This
  exclusion process models conduction in strongly disordered
  one-dimensional media.  We prove that, when varying over the
  disorder and for a suitable slowly varying function $L$, under the
  super-diffusive time scaling $N^{1 + 1/\alpha}L(N)$, the density
  profile evolves as the solution of the random equation $\partial_t
  \rho = \mf L_W \rho$, where $\mf L_W$ is the generalized
  second-order differential operator $\frac d{du} \frac d{dW}$ in
  which $W$ is a double sided $\a$--stable subordinator.  This result
  follows from a quenched hydrodynamic limit in the case that the
  i.i.d. jump rates are replaced by a suitable array $\{\xi_{N,x} :
  x\in\bb Z\}$ having same distribution and fulfilling an a.s.
  invariance principle. We also prove a law of large numbers for a
  tagged particle.

\bigskip


\smallskip

\noindent {\em MSC-class:} Primary 60K35, 60K37. Secondary 82C44.
\end{abstract}

\maketitle

\section{Introduction}
\label{sec0}

Fix a sequence of positive numbers $\xi = \{\xi_x : x \in \bb Z\}$ and
consider the random walk $\{X_t : t\ge 0\}$ on $\bb Z$ which jumps
from $x$ to $x+1$ and from $x+1$ to $x$ at rate $\xi_x$. Assume that
\begin{equation}
\label{g08}
\lim_{\ell\to\infty} \frac 1\ell \sum_{x=1}^\ell \xi_x^{-1} \; =\;
\gamma\;, \quad \lim_{\ell\to\infty} \frac 1\ell \sum_{x=-\ell}^{-1}
\xi_x^{-1} \; =\; \gamma
\end{equation}
for some $0<\gamma <\infty$. It is well known that 
\begin{equation}
\label{g04}
X_{tN}/\sqrt{N} \quad\text{converges in distribution to} \quad
\gamma^{-1} B_t\;,
\end{equation}
as $N\uparrow\infty$, where $B_t$ is a standard Brownian motion (cf.
\cite{KK}, \cite{F} for references).

In the particular case where $\{\xi_x : x\in \bb Z\}$ is an ergodic
sequence of positive random variables (e.g. i.i.d. positive random
variables) with $ E[\xi_0^{-1}]<\infty$, a quenched (i.e.\!  a.s.\!
with respect the environment $\xi$) invariance principle follows from
the previous result setting $\gamma = E[\xi_0^{-1}]$. Notice that the
noise survives in the limit only through the expected value of
$\xi_0^{-1}$. Moreover, in the non trivial case where $\xi$ is not
constant, by Jensen's inequality, the diffusion coefficient
$E[\xi_0^{-1}]^{-1}$ is strictly smaller than the expected value of
the conductance $\xi_x$.

If the positive i.i.d. random variables $\{\xi_x : x\in \bb Z\}$ are
such that $E[\xi_0^{-1}]=\infty$, the invariance principle \eqref{g04}
suggests that the random walk remains freezed in the diffusive scale.
As discussed in \cite{abso}, a natural assumption is to suppose that
the distribution of $\xi_0^{-1}$ belongs to the domain of attraction
of an $\alpha$-stable law, $0<\a<1$. In this case, the partial sums of
the sequence $\{\xi_x^{-1}\}$ converge in law, when properly rescaled,
to a double sided $\a$--stable subordinator $W$: there exists a slowly
varying function $L(\cdot)$ such that, for each $u$ in $\bb R$,
\begin{equation}
\label{g05}
\lim_{N \to \infty} \frac{S (\lfloor Nu \rfloor )}{N^{1/\alpha}L(N)}
\;=\; W(u) \quad\text{in law}\;,
\end{equation}
where $\lfloor a \rfloor = \max \{ k\in \bb Z: k\le a\}$ is the
integer part of $a$, and where the function $S: \bb Z \rightarrow \bb
R$ is defined as
\begin{equation}
\label{g06}
S(j)=
\begin{cases}
\sum _{x=0} ^{j-1} \xi_x ^{-1}\;, & j>0\;, \\
0\;, & j=0\;,\\
-\sum _{x=j} ^{-1} \xi_x ^{-1} \;,& j<0\;.
\end{cases}
\end{equation}

This case is treated in \cite{KK}, where the authors prove that
varying over the environment $\xi$ and for a suitable slowly varying
function $L'(\cdot)$, the process $X ( N^{1+1/\alpha}$ $L'(N)
t)/N$ converges weakly in the Skorohod space $D([0,\infty),\RR)$
to a continuous self-similar process.

In order to prove a quenched limiting behavior we need the
limit \eqref{g05} to be almost sure as well. To transform the
convergence in law into an almost sure convergence we can, for
instance, replace the sequence $\{\xi_x : x\in \bb Z\}$ by an array
$\{\xi_{N,x} : x\in \bb Z\}$, $N\ge 1$, which has the same
distribution as $\{\xi_x : x\in \bb Z\}$ for each $N\ge 1$.  The most
natural array can be constructed as follows. Define the nonnegative
function $G$ on $[0,\infty)$ by $ P ( W(1) >G(x))=P(\xi_0^{-1}>x)$.
Since $W(1)$ has a continuous strictly increasing distribution, $G$ is
well defined, nondecreasing and right continuous. Call $G^{-1}$ the
nondecreasing right continuous generalized inverse of $G$ and set
\begin{equation*}
\xi_{N,x}^{-1} \;=\; G^{-1}\left(
N ^{1 / \a} \{ W (x+1/N) - W(x/N)\} \right)\, .
\end{equation*}
Then (cf.\! \cite[Section 3]{fin}) the array $\{\xi_{N,x} : x\in \bb
Z\}$, defined as function of the subordinator $W$, has the same
distribution of $\{\xi_x : x\in \bb Z\}$ and
\begin{equation*}
\lim _{N\rightarrow \infty}
\frac{ S_N(\lfloor Nu \rfloor) }{N^{1/\a}} = W(u)\;, 
\quad \forall u \in \bb R\;, \quad \text{for a.a. $W$}\;,
\end{equation*}
where $S_N$ is defined as the function $S$ in (\ref{g06}), with
$\xi_x^{-1}$ replaced by $\xi_{N,x}^{-1}$.

Fix $N\ge 1$ and consider the random walk $X_N(t)$ which jumps from
$x$ to $x+1$ and from $x+1$ to $x$ at rate $\xi_{N,x}$. Kawazu and
Kesten proved in \cite{KK} that for a.a.\!  realizations of the
environment and for a suitable slowly varying function $L'(\cdot)$,
$X_N (N^{1+1/\alpha} L'(N) t)/N$ converges weakly to a process $\{Y_t
: t\ge 0\}$. The first main result of this article states that $\{Y_t
: t\ge 0\}$ is a Markov process with continuous paths which is
\emph{not} strongly Markov. In particular, $\{Y_t : t\ge 0\}$ is not a
Feller process. Furthermore, we show in Section \ref{sec2} that the
generator $\mf L_W$ of $\{Y_t : t\ge 0\}$ is the generalized
second-order differential operator
\begin{equation*}
\mf L_W \;=\; \frac d{du} \frac d{dW} \; .
\end{equation*}
We point out that in contrast with the i.i.d. case with $E
[\xi_0^{-1}]<\infty$, the noise $W$ entirely survives in the limit.
In fact, even the generator depends on the realization $W$ of the
subordinator.  \medskip

The second main object of this article is the hydrodynamic behavior of
a one-dimensional simple exclusion process with conductances $\xi =
\{\xi_x : x \in \bb Z\}$.  This is the Markov process on $\{0,1\}^{\bb
  Z}$ which can be informally described as follows. Denote by $\eta$
the configurations of $\{0,1\}^{\bb Z}$ so that $\eta(x) =0$ if site
$x$ is vacant and $\eta(x)=1$ otherwise. We start from a configuration
with at most one particle per site. At rate $\xi_x$ the occupation
variables $\eta(x)$, $\eta(x+1)$ are exchanged. The generator $L$ of
this Markov process acts on local functions $f$ as
\begin{equation}
\label{g4}
L f(\eta) \;=\; \sum_{x \in \bb Z} \xi_x
[f(\sigma^{x,x+1} \eta) - f(\eta) ]\;,
\end{equation}
where $\sigma^{x,x+1} \eta$ is the configuration obtained from $\eta$
by exchanging the variables $\eta(x)$, $\eta(x+1)$:
\begin{equation}
\label{g5}
(\sigma^{x,x+1} \eta)(y) \;=\;
\begin{cases}
\eta (x+1) & \text{ if } y=x,\\
\eta (x) & \text{ if } y=x+1,\\
\eta (y) & \text{ otherwise}.
\end{cases}
\end{equation}

The hydrodynamic behavior of this exclusion process has been derived
in \cite{F}, \cite{jl2} under the law of large numbers \eqref{g08} for
the inverse of the conductances, and previously in \cite{n} under more
restrictive assumptions: Assume \eqref{g08} for some
$0<\gamma<\infty$. Fix a continuous initial profile $\rho_0 : \bb R
\to [0,1]$ and consider a sequence of probability measures $\mu^N$ on
$\{0,1\}^{\bb Z}$ such that
\begin{equation*}
\lim_{N\to\infty}
\mu^N \Big\{ \, \Big\vert \frac 1N \sum_{x\in\bb Z} H(x/N) \eta(x)
- \int H(u) \rho_0(u) du \Big\vert > \delta \Big\} \;=\; 0
\end{equation*}
for every $\delta>0$ and every continuous functions $H$ with compact
support. As proven in \cite{F}, \cite{jl2}, if $\bb P_{\mu^N}$ stands
for the probability measure on the path space $D(\bb R_+, \{0,1\}^{\bb
  Z})$ induced by the initial state $\mu^N$ and the Markov process
speeded up by $N^2$, then for any $t\ge 0$,
\begin{equation*}
\lim_{N\to\infty}
\bb P_{\mu^N} \Big\{ \, \Big\vert \frac 1N \sum_{x\in\bb Z} H(x/N) \eta_t(x)
- \int H(u) \rho(t,u) du \Big\vert > \delta \Big\} \;=\; 0
\end{equation*}
for every $\delta>0$ and every continuous functions $H$ with compact
support. Here $\rho$ is the solution of the heat equation
\begin{equation*}
  \partial_t \rho = \gamma^{-1} \partial^2_u \rho
\end{equation*}
with initial condition $\rho_0$, $t$ stands for the time variable and
$u$ for the macroscopic space variable.

In view of the discussion of the first part of this introduction,
assume that the environment consists of a sequence of i.i.d.  random
variables $\{\xi_x : x \in \bb Z\}$ and that the distribution of
$\xi^{-1}_0$ belongs to the domain of attraction of an $\a$--stable
law, $0<\alpha < 1$. Recall the definition of the array $\{\xi_{N,x} :
x\in \bb Z\}$. By extending the methods developed in \cite{F},
\cite{jl2}, we can prove a quenched hydrodynamic limit for the
exclusion process with random conductance given by the array
$\{\xi_{N,x} : x\in \bb Z\}$. Theorem \ref{mt1} below states that, for
almost all trajectories $W$, the density profile of the exclusion
process with random conductances $\{ \xi_{N,x} : x\in \bb Z \}$
evolves on the time scale $N^{1 + 1/\alpha}$ as the solution of
\begin{equation}
\label{g03}
\partial_t \rho \;=\; \mf L_W \rho \;,
\end{equation}
where $\mf L_W$ is the generalized second order differential operator
defined above.  From this quenched result we deduce in Theorem
\ref{mt2} an annealed result for the original exclusion process with
random conductances given by $\{\xi_x : x\in \bb Z\}$.

The asymptotic evolution of a tagged particle is also examined.  Under
some assumptions on the solution of the hydrodynamic equation
\eqref{g03}, we show that the asymptotic behavior $u(t)$ of a tagged
particle initially at the origin is described by the differential
equation
\begin{equation*}
\frac d{dt+} u (t) \;=\;
\left\{
  \begin{array}{ll}
{\displaystyle - \frac 1{\rho_t(u(t))} \frac {d \rho_t}{dW} (u (t))} &
{\displaystyle \text{if } \frac {d \rho_t}{dW} (u (t)) < 0} \\
{\displaystyle - \frac 1{\rho_t(u (t)-)} \frac {d \rho_t}{dW} (u (t))} &
{\displaystyle \text{if } \frac {d \rho_t}{dW} (u (t)) > 0} \\
\;\;0 & {\displaystyle \text{otherwise}}\;.
  \end{array}
\right.
\end{equation*}
In this formula $\rho_t$ is the solution of \eqref{g03}, the
differential $d/dW$ is defined by \eqref{malditesta}, and $(df /dt+)
(t_0)= \lim_{\epsilon \downarrow 0} \epsilon^{-1} [f(t_0+\epsilon) -
f(t_0)]$.

\section{Notation and results}
\label{sec1}

We state in this section the main results of the article.  In what
follows, for simplicity of notation, we assume that $\{\xi_x^{-1} :
x\in \bb Z\}$ is a sequence of i.i.d.\! non-negative $\a$--stable random
variables, $0<\a<1$, defined on some probability state space $(E, \mc
E, \mf Q)$, i.\! e., we assume that
\begin{equation}
\label{leggestabile}
E_{\mf Q} [ \exp \{- \lambda \xi_x ^{-1 } \}]= \exp [ -c_0 \lambda
^\alpha ]\, , \qquad \lambda >0\,,
\end{equation}
for some positive constant $c_0$.  The reader can check that all the
results and proofs presented below can be easily extended to the
general case where $\xi_x^{-1}$ belongs to the domain of attraction of
an $\a$--stable law.

Let us fix some basic notation: given an interval $I\subset \RR $
and a metric  space $\mathbb{Y}$,  we write  $D(I,\mathbb{Y})$ for
the space of c\`adl\`ag functions $f:I\rightarrow \mathbb{Y}$,
endowed with the Skorohod metric $d_S$ \cite{bil}, \cite{ek}, and we
denote by $\bb D (f)$  the  set of discontinuity points of $f$. If
$\mathbb{Y}=\bb R$, the generalized inverse of $f$ is defined as
\begin{equation*}
f^{-1} (u) =\sup \{ v\in I \,:\, f(v)\leq u \} .
\end{equation*}
Moreover, we denote by $C(\bb Y )$, $C_b (\bb Y )$, $C_c (\bb Y )$,
$C_0 (\bb Y)$ , respectively, the space of continuous real functions
on $\bb Y$, the space of bounded continuous real functions on $\bb Y$,
the space of continuous real functions on $\bb Y$ with compact support
and the space of bounded continuous real functions on $\bb Y$
vanishing at infinity, i.e. such that for any $\e>0$ the function has
modulus smaller than $\e$ outside a suitable bounded subset $U\subset
\bb Y$.

In what follows, we will introduce several processes defined in
terms of the Brownian motion or the  $\a$--stable subordinator.
To this aim, let $B$ be the Brownian motion with $\EE [B (t)^2]=2t$,
 defined on some probability space $\left(\mathbb{X} ,\mathbb{F},
\PP\right)$, and let $L (t,y)$ be the local time of $B$.
 Then,
$\PP$--almost surely,
\begin{equation*}
\int _a ^b L(t,y) dy = \int _0 ^t \mb 1 \{ a\leq B(s)\leq b\} ds
\end{equation*}
for all $t\ge 0$, $a\leq b$. In this formula, $\mb 1\{A\}$ stands
for the indicator function of the set $A$.

Let $W$ be a double sided $\a$--stable subordinator defined on some
probability space $(\Omega, \Ff, P)$ \cite[Section III.2]{b}.  Namely,
$W(0)=0$, $W$ has non-negative independent increments such that for
all $s<t$
\begin{equation*}
E\Big[ \exp \big\{ -\lambda [ W(t)-W(s)] \big\} \Big] \;=\;
\exp\{-c_0\lambda ^\a (t-s)\}
\end{equation*}
for all $\lambda >0$ and the same positive constant $c_0$ as in
\eqref{leggestabile}. The sample paths of $W$ are c\`adl\`ag,
strictly increasing and of pure jump type, in the sense that
\begin{equation*}
W(u) \;=\; \sum_{0<v\le u} \{ W(v) - W(v-)\}\;.
\end{equation*}
The jumps at location $(u,W(u) - W(u-))$ have a Poisson distribution
with intensity $c_0^{1/\a} w^{-\alpha} du\,dw$ on $\bb R \times \bb
R_+$. Given a  realization of the subordinator $W$,   denote by
$\nu$ the Radon measure $d W^{-1}$, so that
\begin{equation*} 
\int f(u) \nu (du) \;=\; \int f( W(u) ) \, du
\end{equation*}
for all $f\in C_c (\RR)$.   The support of $\nu$ is given by
\begin{equation*}
\text{supp}(\nu) \;=\; \overline{W(\RR)} \;=\; \{ W(x),
W(x-)\,:\, x\in \RR\}\;.
\end{equation*}
Finally, given a Radon measure $\mu$ on $\bb R$ and a Borel function
$f$, we set
\begin{equation*}
\int _u ^v f(x) \mu (dx) \;=\;
\begin{cases}
   \int _{(u,v]}f(x) \mu(dx) &  \text{ if } u\leq v\, ,\\
  -\int_{(v,u]} f(x) \mu(dx) & \text{ if  } u>v\,.
 \end{cases}
\end{equation*}

\subsection{Random walk with random conductances}
\label{reginetta}

Given $N\ge 1$, $x\in \ZZ$ and a realization $\{\xi_x : x\in \bb Z\}$
of the environment, consider the random walk $X^\xi _N(t|x)$ on $\ZZ$
having starting point $x$ and generator $\LL_{\xi, N}$ given by
\begin{equation*}
(\LL_{\xi,N} f) (x/N) \;=\;  N^{1+1/\alpha} \, \xi_x 
\, \big\{ f(x+1)-f(x) \big\}  
\;+\; N^{1+1/\alpha} \, \xi_{x-1}\, \big\{ f(x-1) - f(x) \big\}\; .
\end{equation*}

To describe the asymptotic behavior of this random walk, fix a
realization of the subordinator $W$ and set
\begin{equation*}
\psi (t|u) = \int _\RR L(t,v-u) \nu (dv) \;, \quad
\psi ^{-1}(t|u)=\sup\left\{ s\geq 0\,:\, \psi (s|u) \leq t
\right\}\;.
\end{equation*}
It is known \cite[V.2.11]{bg} that
\begin{equation*}
Z(t|u) \;=\; u + B \big( \psi ^{-1}(t|u)\big)\,, \qquad t\geq 0,\,
u\in \overline{W(\RR)}\,,
\end{equation*}
is a strong Markov process on $\overline{W(\RR)}$. Let $Y(t|u)$, with
$t\geq 0$ and $u\in \RR $, be the process defined by
\begin{equation*}
Y(t|u) \;=\; Y_W(t|u) \;=\; W^{-1}\Big( Z(\,t|W(u)\,) \Big)\;.
\end{equation*}
For $u$ in $\bb R$, set $\lceil u \rceil = \min \{ k\in \ZZ\,:\, k\geq
u \}$.  Kawazu and Kesten proved in \cite{KK} that the law of $X^\xi
_N(t|\lceil uN\rceil)/N$ averaged over the environment $\xi$ converges in
distribution to the law of $Y(t|u)$ averaged over $W$. 

We examine in Section \ref{sec2} the process $\{Y_t : t\ge 0\}$.  We
first show that for a.a.\! realizations of the subordinator $W$,
$\{Y_t : t\ge 0\}$ is Markov process with continuous paths which is
not strongly Markov and therefore not Feller.

The definition of the generator $\mf L_W = \frac d{du} \frac d{dW}$ of
the process $\{Y_t : t\ge 0\}$ requires some notation.  Fix a
realization of the subordinator $W$. Denote by $C_{W,b}(\bb R)$ (resp.
$C_{W,0}(\bb R)$) the set of bounded (resp. bounded which vanish at
$\pm \infty$) c\`adl\`ag functions $f:\bb R \to\bb R$ such that $\bb
D(f) \subset \bb D (W)$. $C_{W,0}(\bb R)$ is provided with the usual
sup norm $\|\cdot \|_\infty$. Let $\mf D_W$ be the set of functions
$f$ in $C_{W,0}(\bb R)$ such that
\begin{equation*}
f(x) \;=\; a \;+\; b W(x)\; +\; \int_0 ^x  dW(y) \int_0^y g(z)
dz\;\, \qquad \forall x \in \bb R\;,
\end{equation*}
for some function $g$ in $C_{W,0}(\bb R)$ and some $a$, $b$ in $\bb
R$.  One can check that this function $g$ is unique.
 
Define the linear operator $\mf L_W : \mf D_W \to C_{W,0}(\bb R)$ by
setting $\mf L_W f = g$. Formally,
\begin{equation*}
\mf L_W \;=\; \frac{d}{dx} \frac{d}{dW} \; \cdot
\end{equation*}

Alternatively, one can introduce the generalized derivative $\frac{
d}{dW}$ as follows
\begin{equation}
\label{malditesta}
\frac{d f}{dW} (x) = \lim_{\e\rightarrow 0} \frac{f(x+\e)
-f(x)}{W(x+\e) -W(x)}\;,
\end{equation}
if the above limit exists and is finite. Due to Lemma 0.9 in
\cite{dyn2}[Appendix], given a right continuous function $f$ and a
continuous function $h$, the following identities are equivalent
\begin{align}
& \frac{df}{dW} (x) \;= \;h(x)\;,\qquad \forall x \in \bb R\;,
\label{anna1}\\
& f (b)- f(a)= \int_{(a,b]} h(y) dW (y)\;, \qquad \forall
a<b\;.\label{anna2}
\end{align}
Hence, a function $f\in C_{W,0} (\bb R)$ belongs to $\mf D_W$ if and
only if $\frac{df}{dW}(x)$ is well defined and derivable, and
$\frac{d}{dx}\bigl(\frac{df}{dW}\bigr) \in C_{W,0}(\bb R)$. In this
case
\begin{equation*}
\mf L_W f \;=\; \frac{d}{dx} \frac{d}{dW}f \;=\;
\frac{d}{dx}\left(\frac{df}{dW}\right)\;.
\end{equation*}

We point out that a function $f$ in   $\mf D _W\cap C(\bb R)$  must
be constant, otherwise due to  (\ref{malditesta}) $(df/dW) (x)$
would be $0$ on $\bb D(W)$, which is a dense set of $\bb R$. Since
$df/dW (x)$ is derivable it must be $0$ everywhere. The equivalence
between (\ref{anna1}) and (\ref{anna2}) allows to conclude.

Denote by $\{P_t : t\ge 0\}$ the semigroup of the Markov process $Y_t$
so that for a bounded Borel function $H:\bb R\to \bb R$,
\begin{equation*}
  (P_t H)(u) \;=\; E[H(Y(t|u))] \,, \qquad u \in \RR\;.
\end{equation*}
In Section \ref{sec2} we prove

\begin{theorem}
\label{mt5}
The space $C_{W,0}(\bb R)$ is $P_t$--invariant and $\{P_t : t\geq
0\}$ is a strongly continuous contraction semigroup on $C_{W,0}(\bb
R)$ with infinitesimal generator $\mf L_W = \frac{d}{dx} \frac{d}{dW}$
defined on the domain $\mf D_W$.
\end{theorem}

In particular, given $\rho _0$ in $C_0 (\bb R)$, the function
$\rho_W(t,u) = P_t \rho _0 (u)$, $t\geq 0$, $u$ in $\bb R$, is
continuous in $t$ and c\`adl\`ag in $u$, $ \bb D (\rho _W (t, \cdot))
\subset \bb D (W)$ and $\rho _W (0, u )= \rho _0 (u)$.  Moreover, we
show in \cite{fjl2} that $\rho _W (t,u)$ belongs to $\mf D _W$, $\rho
_W (t,u)$ is strictly positive and
\begin{equation}
\label{g09}
\frac \partial{\partial t} \rho _W (t,u) \;=\;\frac d{du} \frac
d{dW} \rho _W (t,u)
\end{equation}
for all $t>0$.

\subsection{Annealed hydrodynamic limit}
\label{renetta}

Let $\mc X = \{0,1\}^{\bb Z}$ and denote by the Greek letter $\eta$
the configurations of $\mc X$ so that $\eta (x) =0$ if site $x$ is
vacant for the configuration $\eta$ and $\eta (x) =1$ otherwise.

For each fixed realization $\{\xi_x : x\in \bb Z\}$, consider the
exclusion process on $\mc X$ with random conductances $\{\xi_x : x\in
\bb Z\}$. This is the Markov process on $\{0,1\}^{\bb Z}$ with
generator $L$ given by \eqref{g4}.  Given $T>0$ and a probability
measure $\mu$ on $\mc X$, let $\PP^{\xi,N}_{\mu}$ be the law on the
path space $D([0,T], \mc X)$ of the exclusion process $\{ \eta_t :
t\ge 0\}$ with initial distribution $\mu$ and generator $L$
\emph{speeded up} by $N^{1 + 1/\alpha}$. Expectation with respect to
$\PP^{\xi,N}_{\mu}$ is denoted by $\bb E^{\xi,N}_{\mu}$.

Denote by $\mc M = \mc M (\bb R)$ the space of Radon measures on
$\bb R$ endowed with the vague topology, i.e., $\mf m_n \rightarrow
\mf m $  if and only if $ \mf m_n (f)\rightarrow \mf m (f)$ for any
$f\in C_c (\bb R)$.  Let $\p^{N}_{t} \in \mc M$ be the empirical
measure at time $t$. This is the measure on $\bb R$ obtained by
rescaling space by $N$ and by assigning mass $N^{-1}$ to each
particle at time $t$:
\begin{equation}
\label{mis_emp_t}
\pi^{N}_{t} \;=\; \frac{1}{N} \sum _{x\in \ZZ} \eta_t (x)\,
\d_{x/N}\;,
\end{equation}
where $\delta_u$ is the Dirac measure concentrated on $u$. For $H$ in
$C_c(\bb R)$, $\<\pi^N_t, H\>$ stands for the integral of $H$ with
respect to $\pi^N_t$:
\begin{equation*}
\<\pi^N_t, H\> \;=\; \frac 1N \sum_{x\in\bb Z}
H (x/N) \eta_t(x)\;.
\end{equation*}

A sequence of probability measures $\{\mu_N : N\geq 1 \}$ on $\mc X$
is said to be associated to a profile $\rho_0 :\RR\to [0,1]$ if
$\pi^N_0$ converges to $\rho_0 (u) du$, as $N\uparrow\infty$:
\begin{equation*}
\lim _{N\to \infty}
\mu_N \Big\{ \,  \Big | \<\pi^N_0 , H\>  - \int _\RR H(u) \rho_0 (u)
\, du \Big | \; > \; \delta \Big\} \;=\; 0
\end{equation*}
for all $H \in C_c(\RR)$ and for all $\d>0$.

The following theorem describes the hydrodynamic behavior of the
exclusion process with random conductances $\xi_x$:

\begin{theo}
\label{mt2} 
Let $\rho_0:\RR\to [0,1]$ be a uniformly continuous function and let
$\{\mu_N : N\geq 1 \}$ be a family of probability measures on $\mc X$
associated to $\rho_0$. Then, for all $T>0$, all $\d>0$ and all $H \in
C_c(\RR)$,
\begin{equation*}
\lim _{N \to \infty} \int \mf Q (d \xi) \, \bb P^{\xi,N}_{\mu_N}
\Big[ \sup_{0\le t\le T} \Big | \<\pi^N_t , H\> - \int _\RR H(u)
(P^{\xi,N}_t\rho_0) ([u]_N) \, du \Big | \; > \; \delta \Big] \;=\;
0\;,
\end{equation*}
where $P^{\xi, N}_t$ is the Markov semigroup associated to the random
walk $X^\xi _N (t|\cdot)/N$ and $[u]_N=\lceil u N \rceil /N$.
\end{theo}

\begin{corollary}
\label{mt6}
Given a uniformly continuous function $\rho_0:\RR\to [0,1]$, for each
$N\geq 1 $ on a common probability space $(\Theta_N, \mathcal{F} _N,
Q_N)$ one can define a double sided $\a$--stable subordinator $W$ and
an exclusion process $\eta_t$ with law $\int \mf Q (d \xi) \, \bb
P^{\xi,N}_{\mu_N}$ such that, for all $T>0$, all $\d>0$ and all $H \in
C_c(\RR)$,
\begin{equation*}
\lim _{N\to \infty} Q_N \Bigl ( \sup_{0\le t\le T} \Big |
\<\pi^N_t , H\> - \int _\RR H(u) \rho _{W}(t,u)\, du \Big| > \d
\Bigr) \;=\;0\;,
\end{equation*}
where $\rho_W(t,u)=P_t \rho _0(u)$.
\end{corollary}

The proof of Theorem \ref{mt2} is presented in Section \ref{sec3} and
Corollary \ref{mt6} is a straightforward consequence of Theorem
\ref{mt2}. 

The asymptotic behavior of a tagged particle can be recovered from the
hydrodynamic limit of the process. Fix an initial density profile
$\rho_0:\RR\to [0,1]$ in $C_c(\bb R)$ and let $\{\mu_N : N\geq 1 \}$
be a family of probability measures on $\mc X$ associated to $\rho_0$
and with bounded support in the sense that there exists $A>0$ such
that $\mu_N \{\eta(x) = 1\} =0$ for all $|x/N|>A$, $N\ge 1$.  Assume
that the origin is occupied at time $0$. Tag the particle at the
origin and let the process evolve according to the generator
\eqref{g4} speeded up by $N^{1+1/\alpha}$. Denote by $x_t^N$ the
position of the tagged particle at time $t$. Since particles cannot
jump over other particles and since, according to the previous
theorem, the density profile at time $t$ is approximated by
$P^{\xi,N}_t \rho_0$, $x_t^N/N$ must be close to $u^{\xi,N}_t$, the
unique solution of
\begin{equation*}
\int_{-\infty}^{u^{\xi,N}_t} (P^{\xi, N}_t \rho_0)([u]_N) \, du\;=\;
\int_{-\infty}^{0} \rho_0(u) \, du\;.
\end{equation*}
Note that $u^{\xi,N}_t$ is uniquely determined by this equation
because $P^{\xi, N}_t \rho_0$ is strictly positive and, due to Lemma
\ref{s06}, is Lebesgue integrable. In Section \ref{sec3}, we prove

\begin{theorem}
\label{mt4}
Fix an initial density profile $\rho_0:\RR\to [0,1]$ in $C_c(\bb R)$
and let $\{\mu_N : N\geq 1 \}$ be a family of probability measures on
$\mc X$ associated to $\rho_0$, with bounded support and conditioned to
have a particle at the origin. Then, for all $t>0$ and all $\d>0$,
\begin{equation*}
\lim _{N \to \infty} \int \mf Q (d \xi) \, \bb P^{\xi,N}_{\mu_N}
\Big[ \big | x^N_t/N - u^{\xi,N}_t \big | \; > \; \delta \Big] \;=\;
0\;.
\end{equation*}
\end{theorem}

\subsection{Quenched hydrodynamic limit} 
\label{moravia}

For $N\ge 1$, $x$ in $\ZZ$, set
\begin{equation}
\label{moraviabis}
c_x \;=\; c_x (W,N) \;=\;\frac 1{\g_x}\;, \qquad  \g_x\;=\;
\g_x(W,N)\;=\; N ^{1 / \a} \Big\{ W\Big(\frac{x+1}{N}\Big)
-W\Big(\frac{x}{N}\Big)\Big\}\;.
\end{equation}
Note that  $c_x$ equals the constant $ \xi_{x,N}$ defined in the
introduction. Trivially,  $\{ \g_x \,:\, x\in \ZZ \}$ is a family of
i.i.d. $\a$--stable random variables such that
\begin{equation*}
E\Big[  \exp\{ -\lambda \g_x\}\Big] \;=\;
\exp\{-c_0\lambda^\a\}
\end{equation*}
for $\lambda>0$. In particular, for each $N\ge 1$, $\{ \g_x (W,N)
\,:\, x\in \ZZ \}$ has the same distribution as $\{ \xi^{-1}_x \,:\,
x\in \ZZ \}$.

Consider the exclusion process on $\ZZ$ in which we exchange the
occupation variables $\eta(x)$, $\eta(x+1)$ at rate $c_x$. This
is the Markov process on $\mc X$ whose generator $\Ll_N$ acts on local
functions $f:\mc X\to \bb R$ as
\begin{equation}
\label{generare}
(\Ll_N f) (\eta) \;=\; \sum _{x\in \ZZ} c_x
\big\{ f(\sigma^{x,x+1} \eta) - f(\eta)\big\}\; ,
\end{equation}
where $\sigma^{x,x+1} \eta$ is the configuration defined by
\eqref{g5}.

Given $T>0$ and a probability measure $\mu$ on $\mc X$, let
$\PP^{W,N}_{\mu}$ be the law on the path space $D([0,T], \mc X)$ of
the exclusion process $\{ \eta_t : t\ge 0\}$ with initial
distribution $\mu$ and generator $\Ll_N$ \emph{speeded up} by $N^{1
+ 1/\alpha}$.

\begin{theo}
\label{mt1} 
Let $\rho_0:\RR\to [0,1]$ be a continuous function in $C_0(\bb R)$ and
let $\{\mu_N : N\geq 1 \}$ be a family of probability measures on $\mc
X$ associated to $\rho_0$. Then, for almost all $W$ and all $T>0$ the
empirical measure $\p ^{N}_{t}$ converges in probability to the
measure $\rho_W(t,u) du$, where $\rho_W(t,u) = (P_t \rho_0)(u)$:
\begin{equation*}
\lim _{N \to \infty} \bb P^{W,N}_{\mu_N} \Big[ \sup_{0\le t\le T}
\Big | \<\pi^N_t , H\> - \int _\RR H(u) \rho_W (t,u) \, du \Big | \;
> \; \delta \Big] \;=\; 0
\end{equation*}
for all $H \in C_c(\RR)$ and for all $\d>0$.
\end{theo}

As observed above in the annealed case, the asymptotic behavior of a
tagged particle can be recovered from the hydrodynamic limit. 

\begin{theorem}
\label{mt3}
Let $\rho_0:\RR\to [0,1]$ be a continuous function with compact
support. Let $\{\mu_N : N\geq 0 \}$ be a family of probability
measures on $\mc X$ associated to $\rho_0$, with bounded support and
conditioned to have a particle at the origin. Then, for almost all $W$
and for all $t>0$, $x^N_t/N$ converges in $\bb
P_{\mu_N}^{W,N}$-probability, as $N\uparrow\infty$, to $u_W(t)$ given
by
\begin{equation*}
\int_{-\infty}^{u_W(t)} \rho_W(t,u) \, du
\;=\; \int_{-\infty}^{0} \rho_0(u) \, du\;.
\end{equation*}
\end{theorem}

By Proposition \ref{pierpaolo} and Corollary \ref{s05} (iv) with
$H=\rho_0$ and $\rho=1$, $\rho_W(t, \cdot)$ is strictly positive and
integrable so that $u_W(t)$ is uniquely determined. We prove Theorem
\ref{mt3} in Subsection \ref{5.4} and that $u_W(t)$ is continuous.
Moreover, we derive in Lemma \ref{sc1}, under some extra assumptions
on $\rho_W(t, \cdot)$, the following differential equation for $u_W$:
\begin{equation*}
\frac d{dt+} u_W(t) \;=\;
\left\{
  \begin{array}{ll}
{\displaystyle - \frac 1{\rho_t(u_W(t))} \frac {d \rho_t}{dW} (u_W(t))} &
{\displaystyle \text{if } \frac {d \rho_t}{dW} (u_W(t)) < 0} \\
{\displaystyle - \frac 1{\rho_t(u_W(t)-)} \frac {d \rho_t}{dW} (u_W(t))} &
{\displaystyle \text{if } \frac {d \rho_t}{dW} (u_W(t)) > 0} \\
\;\;0 & {\displaystyle \text{otherwise}}\;.
  \end{array}
\right.
\end{equation*}
In this formula $\rho_t(\cdot) = \rho_W(t, \cdot)$, the differential
$d/dW$ is defined by \eqref{malditesta}, and $(df /dt+) (t_0)=
\lim_{\epsilon \downarrow 0} \epsilon^{-1} [f(t_0+\epsilon) -
f(t_0)].$

\section{Quasi-diffusions}
\label{sec2}

In this section we study  the Markov  processes $Y$ and $Z$, defined
in Section \ref{reginetta} in terms of the Brownian motion $B$ and
the subordinator $W$. All the results presented in this section hold
for almost all realizations of the subordinator $W$, although not
always explicitly stated.

Fix a realization of the subordinator $W$. Recall that $\nu=
dW^{-1}$ and that the support of $\nu$ coincides with
$\overline{W(\RR)} = \{ W(x), W(x-)\,:\, x\in \RR\}$.

\subsection{The Markov process  $Z(t|u)$}
\label{giobbe}

We briefly recall some results from \cite{lo2}, \cite{lo1} applied to
the Markov process $Z(t|u)= u+ B\bigl ( \psi^{-1}(t|u)\bigr)$ with
state space $\overline{W(\RR)}$ (see in particular Theorems 1.2.1 and
3.3.1 of \cite{lo2} and Theorem 3.2 of \cite{lo1}). Denote by $\{
Q_t:t\geq 0\}$ the Markov semigroup associated to $Z(t|u)$ acting on
the space $C_0(\overline{W(\RR)})$:
\begin{equation}
\label{petrova}
Q_t f(u) \;=\; \EE\left[ f( Z(t|u))\right]
\end{equation}
for all $f$ in $C_0(\overline{W(\RR)})$.   By endowing
$C_0(\overline{W(\RR)})$ with the uniform norm $\|\cdot\|_\infty$,
$\{ Q_t:t\geq 0\}$ is a strongly continuous semigroup of
contraction operators: for all $f\in C_0 (\overline{W(\RR)})$, $ Q_t
f$ belongs to $C_0 (\overline{W(\RR)})$, $\|Q_t f\|_\infty \leq
\|f\|_\infty$ and
\begin{equation*}
\lim _{s\to 0} \| Q_{t+s} f - Q_t f\|_\infty \; =\; 0\; ,
\end{equation*}
for all $t>0$. The same statement holds for $t=0$ if $s$ takes only
positive values.

Let $D_W$ be the set of functions $f$ in $C_0(\overline{W(\RR)})$ for
which there exists a function $h$ in $C_0(\overline{W(\RR)})$ and $a$,
$b$ in $\bb R$ such that
\begin{equation}\label{mucca}
f(u) \;=\; a \; +\; bu \;+\; \int_0^u dv \int_0 ^v  h(w) \, \nu (dw)
\end{equation}
for all $u$ in $\overline{W(\RR)}$. One can check that $h$ is
univocally determined. We denote $h$ by
\begin{equation*}
h \;=\; \frac{d}{dW^{-1}} \frac d{du} f \;=\; L_W f\;.
\end{equation*}
Then, $L_W: D_W \to C_0(\overline{W(\RR)})$ is the generator of the
Markov semigroup $\{ Q_t:t\geq 0\}$ on $C_0 (\overline{W(\RR)})$.
\medskip

We prove in \cite{fjl2} that the process $Z$ admits a strictly
positive symmetric transition density function $q_t(x,y)$ w.r.t.
$\nu$:

\begin{theorem}
\label{salomone} 
There exists a strictly positive Borel function $q$,
\begin{equation*} 
q: (0,\infty)\times \overline{W(\RR)}\times \overline{W(\RR)}
\rightarrow (0,\infty) \;,
\end{equation*}
symmetric in $x$, $y$, such that
\begin{equation*}
\bb E[ f (Z(t|x))] \;=\; \int f(y) q_t(x,y) \nu (dy)
\end{equation*}
for all $t> 0$, $x$ in $\overline{W(\RR)}$ and $f$ in
$C_{b}(\overline{W(\RR)})$.  Moreover,
\begin{equation*}
\int q_t(x,y) \nu (dy)\;=\; 1
\end{equation*}
for every $t>0$, $x$ in $\overline{W(\RR)}$ and $q_t(\cdot, y) \in D_W$,
\begin{equation*}
\partial_t q_t(\cdot, y) \;=\; L_W q_t(\cdot , y)
\end{equation*}
for every $t>0$ and $y$ in $\overline{W(\RR)}$.
\end{theorem}

\subsection{Markovian properties of the process $Y(t|u)$.}

Denote by $\{x_j : j\ge 1\}$ the jump points of $W$, which form a
countable dense set in $\bb R$.  Since $W^{-1} (W(x_j-))$ $= W^{-1}
(W(x_j))$, the function $W^{-1} : \overline{W(\bb R)} \to \bb R$ is
not one to one and the process $Y(t|u) = W^{-1} (Z(t|W(u)))$ with
space state $\bb R$ could be  a non Markov process. The following
proposition clarifies the Markovian properties of the process $Y$

\begin{prop}
\label{nsf}
The stochastic process $Y$ has continuous paths. It is Markov but not
strongly Markov. In particular, it is not a Feller Markov process.
\end{prop}

\begin{proof}
The continuity of paths can be proved by the same arguments used in
Lemma 4 of \cite{KK}.  In what follows we denote respectively by $\PP
_y ^Y$, $\PP_z ^Z$ the law of $Y(\cdot | y)$ and $Z(\cdot| z)$, where
$y\in \RR$, $z\in \overline{W(\RR)}$, and by $\EE _y^Y $, $\EE _z ^Z$
the related expectations.  Moreover, we write $\o $ for a generic path
in $D( [0,\infty), \RR)$ and $\theta _t \o$ for the time--translated
path $\theta _t \o (s) = \o (t+s)$.  \smallskip

First we prove that $Y$ is a Markov process. To this aim, we fix $y\in
\RR$, $t>0$ and let $\Aa, \Bb \subset D([0,\infty),\RR)$ be of the
form
\begin{eqnarray*}
\!\!\!\!\!\!\!\!\!\!\!\!\! &&
\Aa =\left\{\o\,:\; \o(t_i)\in
[a_i,b_i]\,\forall\, 1\leq i \leq n\right\},\\
\!\!\!\!\!\!\!\!\!\!\!\!\! && \qquad
\Bb =\left\{\o\,:\; \o(s_i)\in [c_i,d_i]\,\forall\, 1\leq i \leq
k\right\},
\end{eqnarray*}
where  $0\leq t_1<t_2<\cdots <t_n<t$ and $0\leq s_1<s_2<\cdots
<s_k$.

Due to the definition of $Y$,
\begin{equation*}
\PP_y ^Y \left ( \o \in \Aa,\, \theta _t\o  \in \Bb\right)
\;=\; \PP ^Z_{ W(y) } \left(  W^{-1}\circ \o\in \Aa  ,  \,
\theta _t (W^{-1}\circ\o) \in \Bb  \right) \;.
\end{equation*}
Since $Z$ is a Markov process, the previous expresssion is
equal to
\begin{equation*}
\EE ^Z _{W(y)} \Big[ \mb 1 \{ W^{-1}\circ \o\in \Aa \}
\PP ^Z _{Z(t|W(y) ) } ( W^{-1}\circ \o\in \Bb )\Big]\,.
\end{equation*}

We claim that $ Z(t|W(y))= W ( Y (t|y))$ with probability $1$.  In
fact, we know that $Z(t|W(y))$ has value $ W( Y(t|y)-)$ or $W(Y(t|y)
)$. If $ W( Y(t|y)-)=W(Y(t|y) )$ the conclusion is trivial. Otherwise,
it must be $ W( Y(t|y)-)= W(x_j-)$ for some $j$. Since a.s. $\nu$ has
no atoms, the countable set $\{W(x_j-):j\ge 1\}$ has zero
$\nu$--measure and due to Theorem \ref{salomone}
\begin{equation*}
\PP\Big[ Z(t| W(y) )\in \{W(x_j -) : j\ge 1\}\Big] \;=\; 0 \;.
\end{equation*}
This allows to conclude that $ Z(t|W(y))= W (Y (t|y))$ with
probability $1$.  Therefore,
\begin{equation*}
\PP ^Z _{ Z(t|W(y) ) } ( W^{-1}\circ \o\in \Bb )
\;=\; \PP ^Z_{ W( Y(t|y) ) } (W^{-1}\circ \o\in \Bb )
\end{equation*}
$\PP ^Z_{W(y)}$--a.s. By definition of $Y$, putting all previous
identities together, we get that
\begin{eqnarray*}
\PP_y ^Y \big[ \o \in \Aa,\, \theta _t\o  \in \Bb\big]
&=& \EE ^Z_{W(y) }\Big[ \mb 1 \big\{  W^{-1}\circ \o\in \Aa \big\}
\, \PP ^Z_{ W( Y(t|y) ) } \big[ W^{-1}\circ \o\in \Bb \big] \, \Big ] \\
&=& \EE_y ^Y \Big[ \mb 1\{ \o \in \Aa\}
\, \PP^Y_ {Y(t|y)} \big[ \o \in \Bb \big]\, \Big]\,.
\end{eqnarray*}
This proves that $Y$ is a Markov process.
\smallskip

 We now show that $Y(t|u)$ is not strongly Markovian with
respect to the filtration $\mc F^Y_t = \sigma (Y_s : s\le t)$. Fix
$y$ in $\bb R$ such that $0<W(y-) < W(y)$ and consider the sets
\begin{eqnarray*}
\!\!\!\!\!\!\!\!\!\!\! &&
A\;=\; \Big\{ \exists \delta > 0 : Y_s \le y \text{ for } 0\le s\le
\delta\Big\}\;,\quad \\
\!\!\!\!\!\!\!\!\!\!\! && \qquad
B\;=\; \Big\{ \exists \delta > 0 : Z_s \le W(y) \text{ for } 0\le s\le
\delta\Big\}\;.
\end{eqnarray*}
Let $\tau $ be the first time the process $Y_t$ reaches $y$: $\tau =
\inf \{ t\ge 0 : Y_t = y\}$ and let $\s$ be the first time the process
$Z_t$ reaches $W(y)$ or $W(y-)$: $\s = \inf \{ t\ge 0 : Z_t = W(y)
\text{ or } Z_t = W(y-)\}$. Since $Y_t$ has continuous paths, $\tau $
is a stopping time for the filtration $\{\mc F^Y_t : t\ge 0\}$ and
$\s$ is a stopping time for the filtration $\{\mc F^Z_t : t\ge 0\}$,
where $\mc F^Z_t = \sigma (Z_s : s\le t)$.

Assume, by contradiction, that $Y$ is strongly Markov. By the strong
Markov property,
\begin{equation*}
\bb P^Y_0 \big[ \theta_\tau  Y \in A \big] \;=\; \bb E^Y_0 \big[ \bb
P^Y_{Y_\tau} \big[ A\big] \, \big]\;.
\end{equation*}
Since $Y_\tau = y$, the previous expectation is equal to $\bb
P^Y_{y} [ A]$. By definition of the process $Y$, this probability
corresponds to $\bb P^Z_{W(y)} [ B]$. This last probability is equal
to $0$ in view of the construction of $Z$ through the Brownian
motion.

On the other hand, by construction of the Markov process $Y(t|0)$,
by definition of the random times $\tau$, $\s$ and since $\s$ is a
stopping time,
\begin{equation*}
\bb P^Y_0 \big[ \theta_{\tau} Y \in A \big] \;=\; \bb P^Z_0 \big[
\theta_{\s} Z \in B \big] \;=\; \bb E^Z_0 \big[ \bb P_{Z_\s} \big[
B\big] \, \big]\;.
\end{equation*}
Since $0<W(y-)$, $Z_{\s} = W(y-)$. In particular, $ P_{Z_\s} $ is
equal to $\bb P_{W(y-)} \big[ B\big]$ and this probability is equal
to $1$ by construction of the process $Z$ through the Brownian
motion.
\smallskip

Finally, $Y$ cannot be a  Feller  Markov process since otherwise it
would be   strong Markov process.
\end{proof}

\subsection{The generator of the process $Y(t |u)$.}\label{siluro}

We obtain in this subsection the generator of the Markov process
$Y$. To keep notation simple, we denote $Z(t|W(u))$ by $Z_t$ and
$Y(t|u)$ by $Y_t$. Moreover, we write $\{x_j\,:\, j\geq 1\}$ for the
jump points of $W$, which form a countable dense set.

Denote by $d_W$ the distance in $\bb R$ defined by $d_W(x,y) = |W(x) -
W(y)|$ and by $\bb R_W$ the completion of $\bb R$ with respect to this
distance. It is easy to check that $\bb R_W$ is obtained by dividing
in two each jump point of $W$: $\bb R_W = \bb R \cup \{x_j^- : j\ge
1\}$ and
\begin{equation*}
d_W(x_j^-, y) \;=\; |W(x_j-) - W(y)|\;, \quad
d_W(x_j^-, x_k^-) \;=\; |W(x_j-) - W(x_k-)|
\end{equation*}
for every $y$ in $\bb R$, $j$, $k\ge 1$.

Recall that $\overline{W(\bb R)} = \{W(x) : x\in \bb R\} \cup
\{W(x_j-) : j\ge 1\}$.  Let
\begin{equation*}
W_e: \bb R_W \to \overline{W(\bb R)}
\end{equation*}
be given by $W_e(x) = W(x)$ for $x$ in $\bb R$, $W_e(x_j^-) = W(x_j-)$
for $j\ge 1$. $W_e$ is an isometry from $(\bb R_W, d_W)$ to
$(\overline{W(\bb R)}, d)$, where $d$ is the usual Euclidean distance
in $\bb R$. Its inverse, $W_e^{-1} : \overline{W(\bb R)} \to \bb R_W$,
is given by $W_e^{-1} (W(x)) = x$, $W_e^{-1} (W(x_j-)) = x_j^-$.
\smallskip

Since $W_e$ is an isometry, all the results concerning the process
$Z_t$ with state space $\overline{W(\bb R)}$ can be trivially restated
in terms of the pullback process $X_t = W_e^{-1}(Z_t)$ with state
space $\bb R_W$. In particular, $X_t$ is a strong Markov process and,
denoting by $\mc Q_t$ its Markov semigroup, $\{\mc Q_t : t\geq 0\}$ is
a strongly continuous semigroup of contraction operators acting on the
space $C_0(\bb R_W)$ with norm $\|\cdot \|_\infty$. Let us describe
its generator $\mc L_W$.

We have seen that the domain of the generator $L_W$ of the Markov
process $Z_t$ is $D_W$. Hence the domain of the generator $\mc L_W$ is
given by the set $\mc D_W$ where $\mc D _W= \{ f\circ W_e: f\in
D_W\}$. Let $f, h\in C(\overline{W(\bb R)})$ be as in (\ref{mucca})
for suitable constants $a,b$. By a change of variables, it simple to
check that the functions $F= f\circ W_e $ and $H=h\circ W_e $,
belonging to $C(\bb R_W)$, satisfy the identity
\begin{equation}
\label{liliana}
F(x) \;=\; a \;+\; b W(x)\; +\; \int_0 ^x  d W(y) \int_0^y H(z) dz
\;, \quad x\in\bb R \;,
\end{equation}
and therefore,  by continuity,
\begin{equation*}
F(x_j^-)= F(x_j-)\;, \qquad \forall j \geq 1\;.
\end{equation*}
Viceversa, if $F , H\in C(\bb R_W)$ fulfill (\ref{liliana}), then $f=
F\circ W_e^{-1}\in C( \overline{W(\bb R)}) $ and $h=H\circ W_e^{-1}\in
C( \overline{W(\bb R)})$ satisfy (\ref{mucca}). In particular, we get
that
\begin{equation*}
\mc D_W\;=\; \left\{ F\in C_0(\bb R_W)\,:\, \exists H\in C_0(\bb R_W),\;
\exists a, b \in \RR\text{ satisfying } (\ref{liliana}) \right\}\;.
\end{equation*}
Moreover, due to the above observations, we obtain that $\mc L_W F =
H$ for all $F\in \mc D _W$ and $H\in C_0(\RR _W)$ as in
(\ref{liliana}).  \medskip

We now turn to the process $Y_t$. Recall the definition of the spaces
$C_{W,0}(\bb R)$ and $\mf D_W$ introduced in Section \ref{reginetta}:
$C_{W,0}(\bb R)$ is the set of functions $F:\bb R \to\bb R$ which are
c\`adl\`ag, whose discontinuities form a subset of $\{x_j : j\ge 1\}$
and which vanish at $\pm\infty$, while $\mf D_W$ denotes the set
of functions $F$ in $C_{W,0}(\bb R)$ such that
\begin{equation*}
F(x) \;=\; a \;+\; b W(x)\; +\; \int_0 ^x  dW(y) \int_0^y G(z) dz\;,
\qquad \forall x \in \bb R\;,
\end{equation*}
for some function $G$ in $C_{W,0}(\bb R)$ and some $a$, $b$ in $\bb
R$.  $G$ is univocally determined and the linear operator $\mf L_W :
\mf D_W \to C_{W,0}(\bb R)$ is defined by setting $\mf L_W F = G$.
Formally, $\mf L_W \;=\; \frac{d}{dx} \frac{d}{dW}$.
\medskip

\noindent{\bf Proof of Theorem \ref{mt5}}.
Since $W^{-1} (W(x_j-))= W^{-1}(W(x_j))$ we can write $ Y_t= \Phi
(X_t)$ where $\Phi: \RR_W\rightarrow \RR$ is defined as $\Phi (x) =x $
for all $x\in \RR$ and $\Phi (x_j^-)= x_j$ for all $j\geq 1$.

Given a function $H:\bb R \rightarrow \bb R$, define the function $\mf
E H: \bb R_W \rightarrow \bb R$ as $\mf E H= H\circ \Phi$.  Viceversa,
given a function $h: \bb R_W \rightarrow \bb R$, define $\mf P h: \bb
R \rightarrow \bb R$ as $\mf P h (x)= h(x)$ for all $x\in \bb R$. One
can easily check that $\mf E$ maps $C_{W,0}(\bb R)$, $\mf D_W$
bijectively onto $C_0 (\bb R_W)$, $\mc D _W$ with inverse function
given by $\mf P$.  Moreover,
\begin{equation}
\label{sirchia1}
\mf L_W H\; =\; \mf P \mc L_W \mf E H\,, \qquad \forall H\in \mf
D_W\;.
\end{equation}
Since  $ Y_t= \Phi (X_t)$, for all $H\in C_{W,0}(\bb R)$ we can
write
\begin{equation*}
P_t H (u) = \EE \big[ H (Y(t|u))\big] = \EE \big[ H\circ \Phi
(X(t|u))\big] =   \mc Q _t \bigl ( \mf E H \bigr) (u)
\;, \qquad \forall  u\in \bb R\,,
\end{equation*}
thus implying  that
\begin{equation}
\label{sirchia2}
P_t H =   \mf P \mc Q_t \mf E H\,, \qquad \forall H \in C_{W,0} (\bb
R)\;.
\end{equation}
Due to the above identity and since $C_0(\bb R_W) $ is $\mc Q
_t$--invariant, the space $ C_{W,0} (\bb R)$ is $P_t$--invariant.
Since $\mf E$ is an isomorphism between the normed spaces $C_{W,0}
(\bb R)$ and $C _0 (\bb R_W) $ (endowed of the uniform norm $\|\cdot
\|_\infty$) and since the identities (\ref{sirchia1}) and
(\ref{sirchia2}) hold, the fact that $\{\mc Q _t : t\geq 0\}$ is a
strongly continuous contraction semigroup on $C _0 (\bb R_W)$ having
generator $\mc L _W$ with domain $\mc D_W$ implies that $\{P_t : t\geq
0\}$ is a strongly continuous contraction semigroup on $C_{W,0} (\bb
R)$ with generator $\mf L_W $ having domain $\mf D_W$.
\qed
\medskip

We conclude this subsection with a result which follows easily from
Theorem \ref{salomone} by a change of variables. It will be
particularly useful in the study of the limiting behavior of the
tagged particle in the exclusion process with random conductances.

%% Recall that we denote by $P_t$ the semigroup of the Markov process
%% $Y_t$. Clearly, by definition of $Y_t$,
%% \begin{equation*}
%% P_t H (u) \;=\; Q_t\left( H \circ W^{-1} \right)
%% \left( W(u) \right)\;.
%% \end{equation*}
%% Moreover,

\begin{proposition}
\label{pierpaolo}
The Borel function $p$ defined on $ (0,\infty)\times \RR \times \RR $
as
\begin{equation*}
p_t(x,y)= q_t(W(x), W(y))
\end{equation*}
is the transition density function of the Markov process $Y$ w.r.t.\!
the Lebesgue measure. $p_t(\cdot, \cdot)$ is a strictly positive
symmetric function and
\begin{equation*}
\int_{\bb R} p_t(x,y) dy\;=\; 1 \;,\quad 
(P_t f)(x) \;=\; \int p_t(x,y) f(y) \, dy\;,
\end{equation*}
for all $t>0$, $x$ in $\bb R$ and functions $f$ in $C_{W,b} (\RR)$.
\end{proposition}

Next result is a consequence of Theorem \ref{mt5} and this
proposition.

\begin{corollary}
\label{s05} Fix a function $H$ in $C_0(\bb R)$.

\renewcommand{\theenumi}{\roman{enumi}}
\renewcommand{\labelenumi}{{\rm (\theenumi)}}

\begin{enumerate}
\item For every $t\ge 0$, $P_tH$ is a c\`adl\`ag function vanishing at
  infinity.

\item If $H$ has compact support, $P_t H$ belong to $L^1(\bb R)$ and
\begin{equation*}
\int _\RR  du\, P_t H  (u) \;=\; \int _\RR du\, H (u)\,.
\end{equation*}

\item As $t\downarrow 0$, $P_tH$ converges to $H$ pointwisely. If $H$
  has compact support this limit takes also place in $L^1(\bb R)$. In
  this case, $P_{t+s} H$ converges to $P_t H$ in $L^1(\bb R)$ as $s\to
  0$. 

\item For any function $H$ in $C_c(\bb R)$ and any function $\rho$
  in $C_b(\bb R)$,
\begin{equation*}
\int du \, (P_t H)(u) \, \rho (u) \;=\;  
\int du \, H(u) \, (P_t \rho) (u)\;. 
\end{equation*}
\end{enumerate}
\end{corollary}

\begin{proof} 
Statement (i) follows from Theorem \ref{mt5} and statement (ii) from
Proposition \ref{pierpaolo}.  The first claim of (iii) follows from
Theorem \ref{mt5}. To prove the second claim of (iii) assume, without
loss of generality, that $H$ is positive and has compact support. By
part (ii) of this corollary, $\int P_tH (u) du$ is equal to $\int H(u)
du$.  In particular, by Scheff\'e Theorem, $P_tH$ converges to $H$ in
$L^1(\bb R)$. It follows from Proposition \ref{pierpaolo} that the
semigroup $P_t$ is a contraction in $L^1(\bb R)$. The third claim of
(iii) follows from this observation and the convergence of $P_tH$ to
$H$ in $L^1(\bb R)$. Claim (iv) follows from the symmetry of the
transition density function $p_t(x,y)$ and Fubini's theorem.
\end{proof}

\section{Random walk with random conductances}
\label{sec4}

Fix a realization of the subordinator $W$ and $N\ge 1$, and recall the
definition of the random variables $\{c_x : x\in \bb Z\}$ given in
(\ref{moraviabis}). We examine in this section the limiting behavior
of the continuous--time random walk on $\bb Z$ which jumps from $x$ to
$x+1$ at rate $c_x$ and from $x$ to $x-1$ are rate $c_{x-1}$.  We
first recall some results due to Stone \cite{S}.  \smallskip

Given a Radon measure $\mu$ on $\RR$ with support, denoted
by $\text{supp}(\mu)$, unbounded from below and from above, for each
$x\in\text{supp}(\mu)$ and $t\geq 0$ set
\begin{equation}
\label{wwf1} \psi (t|x,\mu) = \int _\RR L(t,y-x) \mu (dy) \;, \quad
\psi ^{-1}(t|x,\mu )=\sup\left\{ s\geq 0\,:\, \psi (s|x,\mu ) \leq t
 \right\}\;.
\end{equation}
Then $\psi (\cdot| x,\mu )$ is a continuous function and $ \psi
^{-1}(\cdot|x,\mu )$ is a nondecreasing c\`adl\`ag function. Set
\begin{equation*}
Z(t|x,\mu)= B \left( \psi ^{-1}(t|x,\mu)\right)+x\;.
\end{equation*}
$Z= \{ Z(t|x,\mu) :\, t\geq 0 \} $, defined on probability space
$\left(\mathbb{X},\mathbb{F}, \PP\right)$ as the Brownian motion $B$,
is a strong Markov process with state space $\text{supp}(\mu)$ and
paths in the Skohorod space $D([0,\infty),\RR)$ endowed of the
Skohorod metric $d_S$ \cite[V.2.11]{bg}.

By Theorem 1 and Corollary 1 in \cite{S}, we have

\begin{prop}
\label{pietra} Let $\{\mu_n\}_{n\geq 0} $, $\mu$   be Radon measures
on $\RR$ with support unbounded from below and from above. Suppose
that:
\begin{itemize}
\item $\mu_n\to \mu$ vaguely,
\item if $y_n \in \text{supp}(\mu_n) $ is a converging sequence as
  $n\uparrow\infty$, then $\lim _{n\uparrow\infty}y_n\in
  \text{supp}(\mu)$.
\end{itemize}
Let $x_n\in \text{supp}(\mu) $ be a converging sequence with
$\lim_{n\uparrow\infty}x_n =x$. Then,
\begin{equation*}
\lim_{n\uparrow\infty}
d_S \left( Z (\cdot| x_n, \mu_n), Z(\cdot |x,\mu)\right)=0 
\quad \PP\text{ a.s. }
\end{equation*}
\end{prop}

Let us recall another consequence of the results in  \cite{S} (see
also Section 2 in \cite{KK}):

\begin{prop}
\label{RW} 
Let $\{x_k\}_{k\in \ZZ}$ satisfy $x_k<x_{k+1}$, $\lim _{k\to \pm
  \infty} x_k=\pm \infty$. Fix positive constants $\{w_k\}_{k\in \ZZ}$
and set $\mu =\sum _{k\in \ZZ } w_k \d_{x_k}$. Then $Z(\cdot|x_j,\mu)$
is the continuous--time random walk on $\{x_k\}_{k\in \ZZ}$ starting
in $x_j$ such that after reaching site $x_k$ it remains in $x_k$ for
an exponential time with mean
\begin{equation*}
w_k \frac{ (x_{k+1}-x_k)(x_k- x_{k-1})}{ x_{k+1}-x_{k-1}} 
\end{equation*}
and then it jumps to $x_{k-1}$, $x_{k+1}$ respectively with
probability
\begin{equation*}
\frac{ x_{k+1}-x_k}{ x_{k+1}-x_{k-1} }\,\text{ and }\,
\frac{ x_k-x_{k-1}}{ x_{k+1}-x_{k-1} }\, \cdot
\end{equation*}
\end{prop}

Given $N\ge 1$, $x\in \ZZ$ consider the random walk $X_N(t|x)$ on
$\ZZ$ having starting point $x$ and generator
\begin{equation*}
\LL_N f (x) \;=\;
c_x N^{1+1/\alpha} \{ f(x+1)-f(x) \} \;+\; c_{x-1}
N^{1+1/\alpha} \{ f(x-1) - f(x) \}\; .
\end{equation*}
Denote the transition probabilities of $X_N$ by $p^N$:
\begin{equation}
\label{h1}
p^N_t (x,y) \;=\; P \big[ X_N (t|x) = y \big]
\end{equation}
for $x$, $y$ in $\bb Z$.  By symmetry, $p^N_t (x,y)= p^N_t (y,x)$.

For $N\ge 1$, let $\nu_N$ be the discrete measure defined by
\begin{equation*}
\nu_N = \frac 1N \sum _{x\in \ZZ}  \d_{W(x/N)}\;.
\end{equation*}
As $N\uparrow\infty$, $\nu_N$ converges to $\nu$ vaguely. By
definition of $\nu_N$ and by Proposition \ref{RW} the random walk
$X_N/N$ can be expressed as a space--time change of the Brownian motion:
\begin{equation*}
N^{-1} X_N (\cdot \,|\, x ) \sim W ^{-1}
 \left(Z\left ( \cdot  \,|\, W ( x/ N ),\nu_N\right)\right)\,,
\end{equation*}
where  ``$\sim$'' means that the two processes have the same
law.

Recall that $[u]_N=\lceil uN\rceil/N$ and define
\begin{equation*}
Y_N (t|u)=  W^{-1} \Big( Z\big ( t \,|\, W ([u]_N ),
\nu_N\big)\Big)\;.
\end{equation*}
It follows from the two previous identities that
\begin{equation*}
N^{-1} X_N (\cdot  \,|\, N [u]_N  ) \sim Y_N  (\cdot | u )  \;.
\end{equation*}

\begin{lemma}
\label{ufo}
Let $Y(t|u)= W^{-1}( Z(t\,|\, W(u),\nu ))$.  For all $u\in \RR$,
\begin{equation*}
\lim _{N\to\infty} d_S\left( Y  _N (\cdot |u),
Y (\cdot|u)\right)=0 \quad \mathbb{P}\, a.s.
\end{equation*}
Moreover, for all $u$ in $\bb R$ and for all $T>0$,
\begin{equation*}
\lim _{N\uparrow\infty }  \sup_{0\le t\le T} \big\vert
Y_N (t |u) - Y (t|u) \big\vert \;=\; 0 \quad  \mathbb{P}\,a.s.
\end{equation*}
\end{lemma}

\begin{proof}
If $y_N\in \text{supp}(\nu_N)$ and $y_N\to y \in \RR $ then $y\in
\text{supp}(\nu)$ (see the proof of Lemma 2 in \cite{KK}). Since $\lim
_{N\to \infty } W( [u]_N) = W(u)$ for all $u\in \RR$, and since
$\nu_N$ converges vaguely to $\nu$, by Proposition \ref{pietra},
\begin{equation*}
\lim_{N\to\infty} d_S\left( \,
Z\left(\cdot  \,|\, W ([u]_N ),\nu_N\right),
Z\left(\cdot   \,|\, W(u) ,\nu \right)\,\right)=0,\qquad
\mathbb{P}\text{--a.s.}
\end{equation*}
The first claim of the lemma follows by the same arguments used in the
proof of Proposition 1 in \cite{KK}.  On the other hand, since by
Proposition \ref{nsf}, $Y(\cdot|u)$ has continuous paths $\PP$--a.s.,
the second statement of the lemma follows from the first one.
\end{proof}

Recall that $P_t$ stands for the semigroup of the process $\{Y(t) : t\ge
  0\}$ and let $\{P^N_t:t\ge 0\}$ be the semigroup of the process
  $\{Y_N(t) : t\ge 0\}$. Hence, given a bounded Borel function $H$,
\begin{equation*}
P^N_t H (u) =  \bb E\big[ H\left( Y_N (t|u)\right)\big]\;.
\end{equation*}
It follows from Lemma \ref{ufo} and the dominated convergence theorem
that $P^N_t H$ converges pointwisely to $P_t H$ for every bounded
continuous function $H$ and every $t\ge 0$.

Since $W$ is strictly increasing, $W^{-1}$ is a continuous function.
In particular, since $\lim_{x\to\pm \infty} W(x) = \pm \infty$,
$H\circ W^{-1}$ belongs to $C_c(\RR)$, $C_0(\RR)$ as soon as $H$
belongs.  In the next three lemmata, we prove properties of the
operators $P_t^N$ and $P_t$ and some convergence results of $P_t^N$ to
$P_t$.

\begin{lemma}
\label{s06}
Fix a continuous function $H:\bb R\to\bb R$ with compact support.  For
every $t\ge 0$, $P_t^N H$ belongs to $L^1(\bb R)$ and
\begin{equation*}
\int _\RR du\, P_t^NH (u) \;=\; \frac{1}{N}\sum _{x\in \ZZ}
H (x/N)\;.
\end{equation*}
\end{lemma}

\begin{proof}
Assume without loss of generality that $H \geq 0$.  Since the
transition probability $p_t(x,y)$ is symmetric,
\begin{eqnarray*}
\!\!\!\!\!\!\!\!\!\!\!\!\!\! &&
\int _{\RR} du\,   P_t ^N H  (u) \;=\;
N^{-1} \sum _{x,y\in \ZZ } p^N_{t} (x , y) H (y/N) \\
\!\!\!\!\!\!\!\!\!\!\!\!\!\! && =\;
N^{-1} \sum _{y\in \ZZ} H (y/N) \sum _{x\in \ZZ}
p^N_t(y , x) \;=\;N^{-1} \sum _{y\in \ZZ} H (y/N)\;.
\end{eqnarray*}
This proves the identity and that $P_t^N H$ belongs to $L^1 (\bb R)$.
\end{proof}

\begin{lemma}
\label{ff}
Fix a function $H$ in $C_c (\RR)$ and $t\ge 0$.

\renewcommand{\theenumi}{\roman{enumi}}
\renewcommand{\labelenumi}{{\rm (\theenumi)}}

\begin{enumerate}
\item $P_t^N H$ converges in $L^1(\bb R)$ to $P_t H$.

\item If $H$ has bounded variation in $\bb R$ then, for every $t\ge
  0$, $P_t H$ has also bounded variation.

\item $P_t^N H$ also converges to $P_t H$ with respect to the counting
  measure:
\begin{equation*}
\lim_{N\to\infty} \frac 1N \sum_{x\in\bb Z} \big\vert (P_t^N H)(x/N) -
(P_t H)(x/N) \big\vert \;=\; 0\;.
\end{equation*}

\item For any $\e>0$ there exists $\Psi\in C_c(\RR)$ such that
  \begin{equation*}
\int_\RR du\, | P_t H (u) - \Psi(u) | \;\leq\; \e\; ,\quad
\frac{1}{N} \sum _{x\in \ZZ } | P_tH (x/N)  -\Psi (x/N) |
\;\leq\; \e
  \end{equation*}
for $N$ large enough.
\end{enumerate}
\end{lemma}

\begin{proof} 
Without loss of generality, fix a positive function $H$ in $C_c(\RR)$.
Since $P_t^N H$ converges pointwisely to $P_t H$, in view of Corollary
\ref{s05} (ii) and Lemma \ref{s06}, $P_t^N H$ converges in $L^1(\bb
R)$ to $P_t H$ by Scheff\'e Theorem.

It is easy to couple two copies of the random walk $Y_N$ in such a way
that
\begin{equation*}
Y_N(t|u) \le Y_N(t|v) \,, \qquad \forall t\ge 0\,, \qquad \forall
u\le v\;. 
\end{equation*}
To prove (ii), assume that $H$ is a continuous function of bounded
variation in $\bb R$. Then there exist bounded, continuous, increasing
functions, $H_-$ and $H_+$, such that $H=H_+ - H_-$. By the coupling,
$P_t^N H_\pm$ are bounded increasing functions. Taking the pointwise
limit as $N\uparrow \infty$ of the identity $P_t^N H= P_t^N H_+ - P_t
^N H_-$ we get that $P_t H= P_t H_+ - P_t H_-$ where $P_t H_\pm$ are
bounded increasing functions.  Therefore, $P_t H$ has bounded
variation.

For $N\ge 1$, and a right continuous function $f:\bb R\to \bb R$,
let $T_N f:\bb R\to \bb R$ be given by $(T_N f)(u) = f (\lceil uN
\rceil /N)$.  We claim that $T_N P_t H$ converges in $L^1(\bb R)$ to
$P_tH$. By Corollary \ref{s05} (i), $P_t H$ is right continuous. In
particular, $T_N P_t H$ converges pointwisely to $P_tH$. For $x$ in
$\bb Z$, denote by $V_x$ the total variation of $P_tH$ on $[x,x+1]$.
Let $V:\bb R\to\bb R_+$ be given by $V(u) = V_{\lceil u\rceil
}$. $V$ belongs to $L^1(\bb R)$ because $P_tH$ has bounded variation
due to (ii). Moreover, $T_N P_t H \le P_t H +
V$ which implies that $T_N P_t H$ belongs to $L^1(\bb R)$. By the
dominated convergence theorem, $T_N P_t H$ converges to $P_t H$ in
$L^1(\bb R)$ because $T_N P_t H$ converges pointwisely to $P_tH$.

The sum appearing in (iii) can be rewritten as
\begin{equation*}
\int_{\bb R}du\, \big\vert (P_t^N H)(u) -
(T_N P_t H)(u) \big\vert \;.
\end{equation*}
Since $P_t^N H$ and $T_N P_t H$ converge to $P_t H$ in $L^1(\bb R)$,
statement (iii) follows.

Fix $\varepsilon >0$.  By Corollary \ref{s05} (ii), $P_t H$ belongs to
$L^1(\bb R)$. In particular, there exists a continuous function with
compact support $\Psi$ which approximates $P_t H$ in $L^1(\bb R)$:
$\Vert P_tH - \Psi\Vert_1 \le \varepsilon$.  The sum in (iv) can be
estimated by $\Vert T_N P_tH - P_t H\Vert_1 + \Vert P_tH - \Psi
\Vert_1 + \Vert \Psi - T_N \Psi \Vert_1$. Since $\Psi$ belongs to
$C_c(\bb R)$ and since $T_N P_t H$ converges in $L^1(\bb R)$ to $P_tH$
the first and third term vanish as $N\uparrow\infty$. This proves
(iv).
\end{proof}

For $\lambda>0$, denote by $\{R_\lambda^N : \lambda >0\}$ the
resolvent associated to the semigroup $\{P^N_t : t\ge 0\}$:
\begin{equation*}
R_\lambda^N H  \;=\; \int_0^\infty dt\, e^{-\lambda t}  P_t^N H
\end{equation*}
for $H$ in $C_c(\bb R)$.

\begin{lemma}
\label{s07}
Fix a function $g$ in $C_c(\bb R)$. Then,
\begin{eqnarray*}
\!\!\!\!\!\!\!\!\!\!\!\!\! &&
\lim_{t\to 0} \limsup_{N\to\infty} \frac 1N \sum_{x\in\bb Z}
\big\vert P_t^N g (x/N) - g(x/N) \big\vert \;=\; 0\;, \\
\!\!\!\!\!\!\!\!\!\!\!\!\! && \quad
\lim_{\lambda\to\infty} \limsup_{N\to\infty}
\frac{1}{N} \sum_{x \in \bb Z} \big|\lambda R_\lambda^N g (x/N)
- g (x/N)\big| \;=\; 0 \;.
\end{eqnarray*}
Moreover, for every $\lambda>0$,
\begin{equation*}
\frac{1}{N} \sum_{x \in \bb Z}
\big| \lambda R_\lambda^N g (x/N) \big| \;\le\;
\frac{1}{N} \sum_{x \in \bb Z} \big| g (x/N) \big|\;.
\end{equation*}
\end{lemma}

\begin{proof}
The first expression is bounded above by
\begin{equation*}
\Vert P_t^N g - P_t g \Vert_1 \;+\; \Vert P_t g - g \Vert_1
\;+\; \frac 1N \sum_{x\in\bb Z} \big\vert N \int_{x-1/N}^{x/N}
g (u) du \, -\,  g(x/N) \big\vert\;,
\end{equation*}
where $\Vert \cdot \Vert_1$ stands for the $L_1(\bb R)$ norm. By
Lemma \ref{ff} (i), the first expression vanishes as
$N\uparrow\infty$.  By Corollary \ref{s05} (iii), the second expression
vanishes as $t\downarrow 0$. Since the third expression vanishes as
$N\uparrow\infty$, the first claim of the lemma is proved.

By definition of the resolvent, the second expression is less than
or equal to
\begin{equation*}
\int_0^\infty dt\, \lambda e^{-\lambda t}
\frac 1N \sum_{x\in\bb Z}
\big\vert P_t^N g (x/N) - g(x/N) \big\vert \;.
\end{equation*}
By Lemma \ref{s06} the sum inside the integral is uniformly bounded in
$t$ and $N$. By the first part of this lemma it vanishes as
$N\uparrow\infty$, $t\downarrow 0$. This proves the second claim.

The third claim follows from the definition of the resolvent and Lemma
\ref{s06}.
\end{proof}

\section{Hydrodynamic behavior}
\label{sec3}

 We prove in this section the main Theorems of the article. We
first examine the convergence of the empirical measure $\pi^N$.

Recall that $\mc M$ stands for the space of Radon measures endowed
with the vague topology. Fix a realization of the subordinator $W$ and
$T>0$. For each probability measure $\mu$ on $\{0,1\}^{\bb Z}$, denote
by $\bb Q_\mu^{W,N}$ the measure on the path space $D([0,T], \mc M)$
induced by the measure $\mu$ and the process $\pi^N_t$, introduced in
\eqref{mis_emp_t}, evolving according to the generator
\eqref{generare} speeded up by $N^{1+1/\alpha}$.  Fix a continuous
profile $\rho_0 : \bb R\to [0,1]$ and consider a sequence $\{\mu_N :
N\ge 1\}$ of measures on $\{0,1\}^{\bb Z}$ associated to $\rho_0$. Let
$\bb Q_{W}$ be the probability measure on $D([0,T], \mc M)$
concentrated on the deterministic path $\pi(t,du) = \rho_W(t,u)du$,
where $\rho_W(t,u) = P_t \rho_0$.

\begin{proposition}
\label{s15}
The sequence of probability measures $\bb Q_{\mu_N}^{W,N}$ converges,
as $N\uparrow\infty$, to $\bb Q_{W}$.
\end{proposition}

The proof of this result is divided in two parts. In Section
\ref{ss1}, we show that the sequence $\{\bb Q_{\mu_N}^{W,N} : N\ge
1\}$ is tight and in Section \ref{ss2} we characterize the limit
points of this sequence. 

\subsection{Tightness}
\label{ss1}

Recall that we denote by $\mc M$ the space of positive Radon measures
endowed with the vague topology. In particular, a sequence $\mu_n$ in
$\mc M$ converges to $\mu$ if
\begin{equation*}
\lim_{n\to\infty} \int f \, d\mu_n \; =\; \int f \, d\mu
\end{equation*}
for all continuous functions $f$ with compact support.

This topology can be defined through a metric. It is indeed not
difficult to find a countable family of functions $\{f_j : j\ge 1\}$
in $C_c(\bb R)$ (even in $C^\infty _c (\bb R)$) such that
\begin{enumerate}
\item For each $\epsilon >0$, integer $k\geq 1$ and continuous
  function $f$ with support contained in $[-k,k]$, there exists $j$
  such that $\Vert f_j - f\Vert_\infty \le \epsilon$ and such that the
  support of $f_j$ is contained in $[-k-1,k+1]$.
  
\item For each integer $k\ge 1$, there exists $j$ such that $f_j$ is a
  non-negative function with support contained in $[-k-1, k+1]$ and
  greater or equal to $1$ in $[-k,k]$.
\end{enumerate}

It is easy to check that $\mu_n$ converges to $\mu$ vaguely if and
only if $\int f_j d\mu_n$ converges to $\int f_j d\mu$ for all $j$. In
particular, $\mu_n$ converges to $\mu$ if and only if $d(\mu_n,\mu)$
vanishes, where $d$ is the metric defined by
\begin{equation}
\label{g2}
d(\mu, \mu') \;=\; \sum_{j\ge 1} \frac 1{2^j} \Big\{ 1 \wedge
\Big\vert \int f_j d\mu - \int f_j d\mu' \Big\vert \, \Big \}\;.
\end{equation}
The space $\mc M$ endowed with this metric is a complete separable
metric space.

The closure of a subset $M$ of $\mc M$ is compact if and only if
\begin{equation*}
\sup_{\mu\in M} \mu (K) \;<\; \infty
\end{equation*}
for all compact sets $K$ of $\bb R$. In particular,  if $g_j$ is a
non-negative smooth function with support contained in $[-j-1, j+1]$
and greater or equal to $1$ in $[-j,j]$, a set $M$ such that
\begin{equation*}
  \sup_{\mu\in M} \sum_{j\ge 1} \frac 1{2^j} \int g_j d \mu  \;<\; \infty
\end{equation*}
is compact.
We refer to Section A.10 of \cite{sep} for the proofs of all the
previous statements concerning the vague topology in $\mc M$.
\smallskip

We prove in this subsection that the sequence of probability measures
$\{\bb Q_{\mu_N}^{W,N} : N\ge 1\}$ is tight in $D([0,T], \mc M)$. The
method \cite{jl2} consists in proving relative compactness of an
auxiliary process and then showing that both processes are close.

For $\lambda>0$, let $\{X^{\lambda,N}_t : t\ge 0\}$ be the $\mc
M$-valued Markov process such that
\begin{equation*}
X^{\lambda,N}_t (H) \;=\; \< \pi^N_t, R_\lambda^N H\>\;=\;
\frac{1}{N} \sum _{x\in \bb Z} \bigl (R_\lambda^N H \bigr)(x/N)
\eta_t (x)
\end{equation*}
for every $H$ in $C_c(\bb R)$, where $\{R_\lambda^N : \lambda >0\}$ is
the resolvent associated to the semigroup $\{P^N_t : t\ge 0\}$.  Here
and below, we do not distinguish between a continuous function $H: \bb
R\to \bb R$ and its restriction to $\bb Z_N = \{x/N : x\in\bb Z\}$.

\begin{lemma}
\label{s01}
For each $T>0$ and $\lambda>0$, the sequence of processes
$\{X^{\lambda,N} : N\ge 1\}$ is tight in $D([0,T], \mc M)$.
\end{lemma}

\begin{proof}
By Proposition IV.1.7 in \cite{kl}, it is enough to show that
$\{X^{\lambda,N}_t (g) : 0\le t\le T\}$ is tight in $D([0,T], \bb R)$
for all functions $g$ in $C^1_c(\bb R)$.  Note that the underlying
space in \cite{kl} is compact, while we are working here on $\bb R$.
However, in both cases the topology is given by a metric of type
\eqref{g2} and in both cases the compacts are characterized by
integral of functions. It is easy to adapt the proof of Proposition
IV.1.7 to the present case.

Fix a function $g$ in $C_c^1(\bb R)$ and let $g_\lambda^N =
R_\lambda^N g$.  Note that $g_\lambda^N$ belongs to $\ell^2 (\bb
Z_N)$, the space of square summable functions $f:\bb Z_N\to\bb R$,
because so does $g$. Since $g_\lambda^N$ is the solution of
\begin{equation}
\label{f01}
\lambda g_\lambda^N \;-\; \bb L_N g_\lambda^N \;=\; g\;.
\end{equation}
There is here a slight abuse of notation. We are using the same symbol
$\bb L_N$ for the generator of $X_N(t)$ and the generator of $N^{-1}
X_N(t)$. The context makes clear to which operator we are refering to.
Multiplying \eqref{f01} by $N^{-1} g_\lambda^N(\cdot /N)$ and summing
over $x$, we obtain that
\[
\frac{\lambda}{N} \sum_{x \in \bb Z} g_\lambda^N(x/N)^2
\;+\; \frac{N^{1/\alpha}}{N^2} \sum_{x \in \bb Z} c_x
(\nabla_N g_\lambda^N)(x/N)^2
\;=\; \frac{1}{N} \sum_{x \in \bb Z} g_\lambda^N(x/N) g(x/N),
\]
where $\nabla_N$ stands for the discrete gradient: $(\nabla_N h)(x/N)
= N\{h(x+1/N) - h(x/N)\}$. By Schwarz inequality, we obtain that
\begin{equation*}
\frac{\lambda}{2 N} \sum_{x \in \bb Z} g_\lambda^N(x/N)^2 \;+\;
\frac{N^{1/\alpha}}{N^2} \sum_{x \in \bb Z} c_x (\nabla_N
g_\lambda^N)(x/N)^2 \;\le\; \frac{1}{2\lambda N} \sum_{x \in \bb Z}
g(x/N)^2 \;.
\end{equation*}
Since $g$ is continuous with compact support, $N^{-1} \sum_{x \in \bb
  Z} g(x/N)^2$ is bounded uniformly over $N$.  Hence,
\begin{equation}
\label{silenzioso}
\sup _N \frac{N^{1/\alpha}}{N^2}\sum_{x \in \bb Z} c_x (\nabla_N
g_\lambda^N)(x/N)^2 \leq \frac{C(g) }{\l }\;.
\end{equation}
 In this formula and below, $C(g)$ stands for some finite
constant depending only on $g$.

An elementary computation shows that $\mc L_N \<\pi^N_t, H\> =
\<\pi^N_t, \bb L_N H\>$,  where  $\mc L_N$ is the operator defined
in (\ref{generare}). In particular,  the process $M_t^{N,\lambda}$
defined by
\begin{equation}
\label{f02}
M_t^{N,\lambda} \;=\; X^{\lambda,N}_t (g) \;-\; X^{\lambda,N}_0 (g)
\;-\; \int_0^t ds\, \big\{ \lambda \< \pi_s^N , g_\lambda^N\>
- \< \pi_s^N , g\> \big\}
\end{equation}
is a martingale with quadratic variation
\begin{equation*}
\<M^{N,\lambda}\>_t \;=\; \int_0^t ds\, \frac{N^{1/\alpha}}{N^3}
\sum_{x \in \bb Z}   c_x (\nabla_N g_\lambda^N)(x/N)^2
\{\eta_s(x+1) -\eta_s(x)\}^2 \;.
\end{equation*}
Due to  the previous estimate (\ref{silenzioso}), the quadratic
variation $\<M^{N,\lambda}\>_t$ satisfies
\begin{equation}
\label{artemisia}
\<M^{N,\lambda}\>_t \le C(g) T /\lambda N\;, \qquad \forall \,0\leq
t \leq T\;.
\end{equation}

Moreover, by Lemma \ref{s07},
\begin{equation} 
\label{rasputin}
\sup_{N\ge 1} \frac{\lambda }{N}
\sum_{x \in \bb Z} \big| g_\lambda^N (x/N) \big| \;\le \; C(g)\;.
\end{equation}
Hence, given $N\geq 1$ and constants $0\leq a<b$,
\begin{equation}
\label{tardi}
\Big | \int_a^b  ds\, \big\{ \lambda \< \pi_s^N , g_\lambda^N\>
- \< \pi_s^N , g\> \big\}\, \Big| \;\leq\; C(g) (b-a)\;.
\end{equation}

\smallskip 
Due to the decomposition (\ref{f02}), to the estimates
(\ref{artemisia}) and (\ref{tardi}), and to Doob inequality, it is
simple to prove (cf. \cite{kl} Chap. IV, p. 55) that, given $\e>0$,
\begin{equation}
\label{translucenza}
 \lim _{\g\downarrow 0} \limsup _{N\uparrow \infty} \sup
_{\t, \theta} \PP _{\mu^N} \left(  \bigl|X^{\lambda,N}_\t(g)-
X^{\lambda,N}_{\t+\theta}(g) \bigr|>\e \right) \;=\;0\;,
\end{equation}
 where the supremum $\sup _{\t, \theta} $ is over all
stopping times $\t$ bounded by $T$, and over $\theta$ with $0\leq
\theta \leq \gamma$.

 In addition, due to (\ref{rasputin}), 
\begin{equation}
\label{ras1}
\sup _{t\geq 0} \sup_{N\geq 1} \bigl| X_t^{\l,N}(g)\bigr|
\;\le\; \frac{C(g)}\lambda \;\cdot
\end{equation}
As discussed in \cite{kl}, Chap. IV, p. 51, (\ref{translucenza}) and
(\ref{ras1}) allow to apply Prohorov's theorem, thus concluding the
proof of the tightness of $X^{\lambda,N}_t (g)$.
\end{proof}

\begin{corollary}
\label{s02}

The sequence of measures $\{\bb Q_{\mu_N}^{W,N} : N\ge 1\}$ is tight.
\end{corollary}

\begin{proof}
It is enough to show that for every function $g$ in $C^1_c(\bb R)$
and every $\epsilon>0$, there exists $\lambda>0$ such that
\begin{equation*}
\lim_{N\to\infty} \bb P^{W,N}_{\mu^N} \Big[
\sup_{0\le t\le T} |X^{\lambda,N}_t (\lambda g) - 
\<\pi^N_t, g\>| > \epsilon
\Big] \;=\;0
\end{equation*}
because in this case the tightness of $\pi^N_t$ follows from the
tightness of $X^{\lambda,N}_t$.  Since there is at most one particle
per site the expression inside the absolute value is less than or
equal to
\begin{equation*}
\frac{1}{N} \sum_{x \in \bb Z} \big|\lambda g_\lambda^N (x/N)
- g (x/N)\big|\;.
\end{equation*}
By Lemma \ref{s07} this expression vanishes as $N\uparrow\infty$,
$\lambda\uparrow\infty$.
\end{proof}

\subsection{Proof of Theorem \ref{mt1} }
\label{ss2}

We start this subsection with a generalization of \cite{n}, \cite{F}.

\begin{lemma}
\label{gaetano}
Fix a function $H$ in $C_c(\bb R)$ and a sequence of probability
measures $\{\mu_N : N\ge 1\}$ in $\{0, 1\}^{\bb Z}$. For each $t\ge
0$,
\begin{equation*}
\lim_{N\to\infty} \bb E_{\mu_N}^{W,N} \Big[ \Big\{ \frac{1}{N}
\sum_{x\in \ZZ} H(x/N) \, \eta_t (x) \;-\; \frac{1}{N} \sum_{x\in
\ZZ} (P_t^N H) (x/N)\, \eta_0(x) \Big\}^2 \Big] \;=\; 0\;,
\end{equation*}
where  $\bb E_{\mu_N}^{W,N}$ denotes the expectation w.r.t. $\bb
P_{\mu_N} ^{W,N}$.
\end{lemma}

\begin{proof}
We assume that the number of particles is finite $\mu_N$--a.s. One can
extend the proof to the general case by the same arguments used in
\cite{F}.  For each $x$ in $\bb Z$, denote by $J^{x,x+1}_t$ the net
current of particles through the bond $\{x,x+1\}$ in the time interval
$[0,t]$.  This is the total number of particles which jumped from $x$
to $x+1$ minus the total number of particles which jumped from $x+1$
to $x$ in the time interval $[0,t]$. Denote by $M^{x,x+1}_t$ the
martingale associated to $J^{x,x+1}_t$: $M^{x,x+1}_t = J^{x,x+1}_t -
N^{1+ 1/\alpha} c_x \int_0^t \{ \eta_s(x) - \eta_s(x+1) \} \, ds$ and
let $M_t^x = M^{x-1,x}_t - M^{x,x+1}_t$.

With this notation,
\begin{equation*}
\eta_t(x) \;=\; \sum_{y\in\bb Z} p^N_t(x,y) \eta_0(y) \;+\;
\sum_{y\in\bb Z} \int_0^t p^N_{t-s}(x,y) \, dM^y_s \;,
\end{equation*}
 where $p^N_t(\cdot, \cdot)$ is the transition probability defined
in \eqref{h1}.
In particular, since $p^N_t$ is symmetric, the expression inside the
square in the statement of the proposition can be rewritten as
\begin{equation*}
\Gamma^N_t \;=\;
\frac 1N \sum_{y\in\bb Z} \int_0^t (P_{t-s}^N H) (y/N) \, dM^y_s\;.
\end{equation*}
To prove the proposition it is therefore enough to show that $\bb
E_{\mu_N}^{W,N} [(\Gamma^N_t)^2]$ vanishes as $N\uparrow\infty$. Since
the martingales $M^{z,z+1}_t$ are orthogonals, $\bb E_{\mu_N}^{W,N}
[\Gamma_t^2]$ is equal to
\begin{equation*}
\frac {N^{1/\alpha}}N
\sum_{y\in\bb Z} \int_0^t ds\, \Big\{ (P_{t-s}^N H) ((y+1)/N) -
(P_{t-s}^N H) (y/N) \Big\}^2 c_{y} \, \bb E_{\mu_N}^{W,N} \big[ a_{y,
y+1} (\eta_s)\big] \;,
\end{equation*}
where $a_{z,z+1}(\eta) = \{\eta(z) - \eta(z+1)\}^2$. Since this
function is bounded by $1$, the previous expression is less than or
equal to
\begin{equation*}\label{uffina}
\frac {N^{1/\alpha}}N \sum_{y\in\bb Z} \int_0^t ds\, c_{y} \Big\{
(P_{s}^N H) ((y+1)/N) - (P_{s}^N H) (y/N) \Big\}^2 \;.
\end{equation*}
Let $H_t(x/N) = (P_{t}^N H) (x/N)$ and observe that $\partial_t H_t
= \bb L_N H_t $. Hence the previous expression can be rewritten as
\begin{multline*}
-\frac{1}{N^2}  \int _0 ^t (H_s, \bb L_N H_s)_Nds\,=\,-\frac{1}{2
N^2} \int _0 ^t \partial _s (H_s,
H_s)_N ds\,= \\
 \frac{1}{2 N^2}    \sum_{y\in\bb Z}  H^2 (y/N) \;-\;
\frac{1}{2 N^2}
  \sum_{y\in\bb Z} \big\{ (P_{t}^N H) (y/N) \big\}^2\;,
\end{multline*}
where $(\cdot,\cdot)_N$ denotes the scalar product on $\ell^2(\bb Z_N)
$ w.r.t. the counting measure. This proves the lemma.
\end{proof}

We are now in a position to show that the sequence of probability
measures $\bb Q_{\mu_N}^{W,N}$ converges, as $N\uparrow\infty$, to
$\bb Q_{W}$.

\smallskip
\noindent{\bf Proof of Proposition \ref{s15}}.
By Corollary \ref{s02}, the sequence $\{\bb Q_{\mu_N}^{W,N} : N\ge
1\}$ is tight. To prove the lemma we only need to characterize the
limit points of this sequence.

Fix a function $H$ in $C_c(\bb R)$. We claim that
\begin{equation}
\label{fisso}
\lim _{N \to \infty} \bb P^{W,N}_{\mu_N} \Big[ \,
\Big | \<\pi^N_t , H\> - \int _\RR H(u) \rho_W (t,u) \, du \Big | \;
> \; \delta \, \Big] \;=\; 0
\end{equation}
for all $0\le t\le T$, $\d>0$.

By Lemma \ref{gaetano} we only need to prove that
\begin{equation*}
\lim _{N\uparrow\infty} \mu_N \Big ( \Big |
\frac{1}{N} \sum _{x \in \ZZ} P_t^N H (x/N) \eta(x)  -
\int H(u) \rho_W (t,u) du \Big| > \d \Big) \;=\; 0\;.
\end{equation*}
Since there is at most one particle per site, by Lemma \ref{ff} (iii),
we may replace $P_t^N H$ by $P_t H$. By assumption on $\mu_N$ and by
Lemma \ref{ff} (iv), we may also replace $N^{-1} \sum_{x\in\bb Z}
(P_t H)(x/N) \eta (x)$ by $\int (P_t H)(u) \rho_0(u) du$. Since
$\rho_0$ is continuous and bounded, by Corollary \ref{s05} (iv), this
expression is equal to
\begin{equation*}
\int H (u) (P_t \rho_0) (u) du \;=\; \int H (u) \rho_W(t,u) \, du\;.
\end{equation*}
This concludes the proof of (\ref{fisso}).

By \eqref{fisso}, the finite dimensional distributions (f.d.d.) of
$\bb Q_{\mu_N}^{W,N}$ converge to the f.d.d. of $\bb Q_{W}$.  Since the
f.d.d. characterize the measure, the proposition is proved. \qed
\smallskip 

\noindent{\bf Proof of Theorem \ref{mt1}.} 
Fix a function $H$ in $C_c(\bb R)$. On the one hand, by Theorem
\ref{mt5} and Corollary \ref{s05} (iv), $\int _\RR H(u) \rho_W (t,u)
\, du$ is a bounded continuous function of time. On the other hand,
for any continuous function $A: [0, T] \to \bb R$, the map in
$D([0,T], \mc M)$
\begin{equation*} 
\pi \;\to \; \sup_{0\le t\le T} \Big | \<\pi_t , H\> - A(t) \Big |
\end{equation*}
is bounded and continuous for the Skorohod topology. In particular,
Theorem \ref{mt1} follows from Proposition \ref{s15}.  \qed

\subsection{Proof of Theorem \ref{mt2}}

We first claim that
\begin{equation}
\label{c5}
\lim_{N\to\infty} \sup_{0\le t\le T}  \Big | \int _\RR H(u) \Big\{
P_t \rho_0 (u) - P^N_t \rho_0 (u) \Big\} \, du \Big |
\;=\;0
\end{equation}
for every function $H$ in $C_c(\bb R)$. Fix $\varepsilon >0$ and
recall that we denote by $(\bb X, \bb F, \bb P)$ the probability
space in which the processes $Y(t|u)$ and $Y_N(t|u)$ are defined.
Expectation with respect to $\bb P$ is denoted by $\bb E $. If $K$
stands for a compact subset of $\bb R$ which contains the support of
$H$, the previous supremum is bounded above by
\begin{equation*}
C_0(H) \sup_{0\le t\le T}  \int_K \bb E  \Big[\, \big |  \rho_0
(Y(t|u)) - \rho_0 (Y_N(t|u)) \big | \, \Big]
 \, du \;.
\end{equation*}
Since $\rho_0$ is uniformly continuous, there exists $\delta>0$ for
which the previous is less than or equal to
\begin{equation*}
C_0(H) \varepsilon \; +\; C_0(H, \rho_0)
\int_K \bb P \Big[ \sup_{0\le t\le T} \big |  Y(t|u) - Y_N(t|u) \big | >
\delta \Big]  \, du \;.
\end{equation*}
By Lemma \ref{ufo} and the dominated convergence theorem, the second
expression vanishes as $N\uparrow\infty$ for every $\delta>0$. This
proves Claim \eqref{c5}. 

It follows from \eqref{c5} and Theorem \ref{mt1} that
\begin{equation*}
\lim _{N \to \infty} \bb P^{W,N}_{\mu_N} \Big[ \sup_{0\le t\le T}
\Big | \<\pi^N_t , H\> - \int _\RR H(u) P^N_t \rho_0 (u) \, du \Big
| \;
> \; \delta \Big] \;=\; 0\;.
\end{equation*}

Since, for each $N\ge 1$, $\{ \g_x (W,N) \,:\, x\in \ZZ \}$ has the
same distribution as $\{ \xi^{-1}_x \,:\, x\in \ZZ \}$,
\begin{eqnarray*}
&& \bb P^{W,N}_{\mu_N} \Big[ \sup_{0\le t\le T}  \Big | \<\pi^N_t ,
H\>
- \int _\RR H(u) P^N_t \rho_0 (u) \, du \Big | \; > \; \delta \Big] \\
&&  \quad \;=\;
\bb P^{\xi,N}_{\mu_N} \Big[ \sup_{0\le t\le T}  \Big | \<\pi^N_t , H\>
- \int _\RR H(u) P^{N,\xi}_t \rho_0 (u) \, du \Big | \; > \; \delta \Big]
\end{eqnarray*}
in distribution. In particular, Theorem \ref{mt2} follows from
Theorem \ref{mt1}. \qed \smallskip

\subsection{The tagged particle}
\label{5.4}
 
We examine in this subsection the asymptotic behavior of the tagged
particle.  As mentioned in Section \ref{sec1}, the law of large
numbers for the tagged particle in the case of compactly supported
initial density profiles is a direct consequence of the hydrodynamic
limit and the fact the relative order among particles is preserved by
the dynamics.  \smallskip

We first prove that the position of the tagged particle, $u_W(t)$, is
uniquely determined.

Recall the notation introduced in Section \ref{sec1} and assume that
the initial density profile $\rho_0$ belongs to $C_c(\bb R)$. It
follows from Proposition \ref{pierpaolo} that
\begin{equation*}
\rho_W(t,u) \;=\; (P_t \rho_0) (u)
\;=\; \int p_t(u,v) \rho_0(v) \, dv\;.
\end{equation*}
Since $p_t(u,v)$ is strictly positive, $\rho_W(t, \cdot)$ is strictly
positive as soon as $\rho_0$ is not identically equal to $0$. On the
other hand, since $P_t$ is a contraction in $L^1(\bb R)$, $\rho_W(t,
\cdot)$, $t\ge 0$, belongs to $L^1(\bb R)$. In particular, for each
$s\ge 0$, there exists a unique $u_W(s)$ in $\bb R$ such that
\begin{equation}
\label{c2}
\int_{-\infty}^{u_W(s)} \rho_W(s,v) \, dv \;=\; \int_{-\infty}^{0}
\rho_0(v) \, dv\; .
\end{equation}

The function $u_W(t)$ is continuous in time. Indeed, on the one hand,
it follows from \eqref{c2} with $s=t$, $t_n$, that 
\begin{equation*}
\Big\vert \, \int_{u_W(t)}^{u_{W} (t_n)} \rho_W(t,v) \, dv \, \Big\vert
\;\le\; \int_{\bb R} dv\, \big\vert  \rho_W (t,v) 
- \rho_W(t_n,v) \big\vert \;.
\end{equation*}
On the other hand, by Corollary \ref{s05} (iii), $\rho_W(t_n, \cdot)$
converges in $L^1(\bb R)$ to $\rho_W(t, \cdot)$ if $t_n\to t$.  Since
$\rho_W(t,\cdot)$ is strictly positive, $u_W(t_n)$ must converge to
$u_W(t)$. \smallskip

\noindent{\bf Proof of Theorem \ref{mt3}}.
Fix a density profile $\rho_0$ in $C_c(\bb R)$ and let $\rho_W(t,u) =
P_t \rho_0$. Observe that $\rho_W(t, \cdot)$ belongs to $L^1(\bb R)$
in virtue of Proposition \ref{pierpaolo}.  Consider a sequence
$\{\mu_N : N\ge 1\}$ of measures associated to $\rho_0$, conditioned
to have a particle at the origin and such that $\mu_N\{\eta(x) = 1\}
=0$ for $|x/N|$ large enough.  For $a$ in $\bb R$, denote by $H_a$ the
indicator function of the interval $[a,\infty)$: $H_a(u) = \mb
1\{[a,\infty)\} (u)$. We first claim that Theorem \ref{mt1} can be
extended to such test functions:
\begin{equation*}
\lim _{N \to \infty} \bb P^{W,N}_{\mu_N} \Big[ \sup_{0\le t\le T}
\Big | \<\pi^N_t , H_a\> - \int _a^\infty \rho_W (t,u) \, du \Big | \;
> \; \delta \Big] \;=\; 0
\end{equation*}
for all $\d>0$, $a$ in $\bb R$. The same statement holds for $\check
H_a = 1-H_a$ in place of $H_a$.

Indeed, consider a sequence of compactly supported continuous
functions $G_k$ (resp. $\check G_k$) increasing to $H_a$ (resp.
$\check H_a$). On the one hand, $\<\pi^N_t , H_a\> \ge \<\pi^N_t,
G_k\>$. On the other hand, $\<\pi^N_t , H_a\> \le N^{-1} \sum_x
\eta_t(x) - \<\pi^N_t, \check G_k\>$. Since the total number of
particles is conserved, this expression is equal to $N^{-1} \sum_x
\eta_0 (x) - \<\pi^N_t, \check G_k\>$. To conclude the proof of the
claim it remains to let $N\uparrow\infty$ and then $k\uparrow\infty$
and to recall that $\int \rho_0(u) du = \int \rho_W(t,u) du$.

We are now in a position to prove Theorem \ref{mt3}. Fix $\delta >0$
and assume that $x^N_t/N \ge u_W(t) + \delta$. Since the total numbers
to the right of the tagged particle doesn't change in time,
$\sum_{x\ge 0} \eta_0(x) = \sum_{x\ge x^N_t} \eta_t(x) \le
\sum_{x/N\ge u_W(t) + \delta} \eta_t(x)$. Dividing by $N$ and letting
$N \uparrow\infty$, by the previous observation, we get that
\begin{equation*}
\int_0^\infty \rho_0(u) du \;\le\; \int_{u_W(t) + \delta}^\infty
\rho_W(t,u) du\;.
\end{equation*}
This contradicts the definition of $u_W(t)$ because $\rho_W(t,\cdot)$
is strictly positive in view of Proposition \ref{pierpaolo}.
Similarly, one can prove that the event $ x_t^N /N \leq u_W(t)-
\delta$ has negligible probability as $N\uparrow \infty$. \qed
\smallskip

\noindent{\bf Proof of Theorem \ref{mt4}.}
Let $u^{W,N}_t$, the unique solution of
\begin{equation*}
\int_{-\infty}^{u^{W,N}_t} (P^{N}_t \rho_0)([u]_N) \, du\;=\;
\int_{-\infty}^{0} \rho_0(u) \, du\;.
\end{equation*}
Note that $u^{\xi,N}_t$ is uniquely determined by this equation
because $P^{N}_t \rho_0$ is strictly positive and, due to Lemma
\ref{s06}, is Lebesgue integrable.

As in the proof of Theorem \ref{mt2}, since $\{ \g_x (W,N) \,:\, x\in
\ZZ \}$ has the same distribution as $\{ \xi^{-1}_x \,:\, x\in \ZZ
\}$, the random variables
\begin{equation*}
\bb P^{\xi,N}_{\mu_N}
\Big[ \big | x^N_t/N - u^{\xi,N} (t) \big | \; > \; \delta \Big]
\quad\text{and}\quad 
\bb P^{W,N}_{\mu_N}
\Big[ \big | x^N_t/N - u^{W,N} (t) \big | \; > \; \delta \Big]
\end{equation*}
depending on $\xi$ and $W$, respectively, have the same law. 

We claim that for each $t>0$ and realization $W$, $u^{W,N} (t)$
converges to $u_W(t)$ as $N\uparrow\infty$. Indeed, since
$\int_{-\infty}^{u_W(t)} \rho_W(t,v) \, dv =  \int_{-\infty}^{u^{W,N}
  (t)} (P^N_t \rho_0)(v)\, dv$, 
\begin{equation*}
\Big\vert \, \int_{u_W(t)}^{u^{W,N} (t)} \rho_W(t,v) \, dv \, \Big\vert
\;\le\; \int_{\bb R} dv\, \big\vert (P^N_t \rho_0)(v) 
- (P_t \rho_0)(v) \big\vert \;.
\end{equation*}
By Lemma \ref{ff} (i), the right hand side vanishes as
$N\uparrow\infty$. Since, by Proposition \ref{pierpaolo}, $P_t \rho_0$
is strictly positive, the claim is proved.

In particular, by Theorem \ref{mt1}, for all $t>0$, $\delta>0$ and
a.a. $W$, 
\begin{equation*}
\lim_{N\to\infty} \bb P^{W,N}_{\mu_N}
\Big[ \big | x^N_t/N - u^{W,N} (t) \big | \; > \; \delta \Big]\;=\;0\;.
\end{equation*}
Theorem \ref{mt4} follows from the dominated convergence theorem, from
the first observation of the proof and from this last observation.
\qed\medskip

We conclude this section deriving a differential equation for the
asymptotic position of the tagged particle. Recall the definition of
the derivative $d/dW$ introduced in \eqref{malditesta} and let $\rho_t
(\cdot) = \rho_W(t, \cdot) = P_t \rho_0$.

\begin{lemma}
\label{sc1}
Fix a profile $\rho_0 : \bb R \to [0,1]$. Assume that
$\rho_0$ belongs to $C_c (\bb R)$ and that
\begin{equation*}
\lim_{h\to 0} \int_{\bb R} \Big\vert \frac{P_h \rho_t - \rho_t}h
- \mf L_W \rho_t\Big\vert \;=\; 0
\end{equation*}
for all $t>0$.  Then, $u_W$ is differentiable both from the right and
from the left for $t>0$ and
\begin{equation*}
\frac d{dt+} u_W(t) \;=\;
\left\{
  \begin{array}{ll}
{\displaystyle - \frac 1{\rho_t(u_W(t))} \frac {d \rho_t}{dW} (u_W(t))} &
{\displaystyle \text{if } \frac {d \rho_t}{dW} (u_W(t)) < 0} \\
{\displaystyle - \frac 1{\rho_t(u_W(t)-)} \frac {d \rho_t}{dW} (u_W(t))} &
{\displaystyle \text{if } \frac {d \rho_t}{dW} (u_W(t)) > 0} \\
\;\;0 & {\displaystyle \text{otherwise}}\;.
  \end{array}
\right.
\end{equation*}
\end{lemma}

Notice that while $\rho_t =P_t \rho_0$, which belongs to the domain of
the generator $\mf L_W$ in view of \cite{fjl2}, may have
discontinuities at $\{x_j :\, j\ge 1\}$, its derivative $d \rho_t/dW$
is continuous.  On the other hand, the definition of the infinitesimal
operator $\mf L_W$ and the fact that $\rho_t$ belongs to the domain of
the generator imply the uniform convergence of $h^{-1}(P_h \rho_t -
\rho_t)$ to $\mf L_W \rho_t$ on the whole real line. We are requiring
here the convergence to take place in $L^1(\bb R)$.  In particular,
$\mf L_W \rho_t$ belongs to $L^1(\bb R)$.

\begin{proof} 
Fix $t>0$ and take $h$ in $\bb R$ small.  Denote $\mf L_W \rho_t$ by
$\lambda_t$.  Since $\int_{(-\infty, u_W(t)]} \rho_t(v)$ $dv =
\int_{(-\infty, u_W(t+h)]} \rho_{t+h}(v) dv$,
\begin{equation*}
- \frac 1h \int_{u_W(t)}^{u_W(t+h)} \rho_t(v) \, dv
\;=\; \int_{-\infty}^{u_W(t+h)} 
\frac{\rho_{t+h}(v) - \rho_t(v)}h \, dv \;.
\end{equation*}
By the main assumption of the lemma, we may replace on the right hand
side the ratio $\{\rho_{t+h}(v) - \rho_t(v)\}/h$ by $\lambda_t$ paying
a price of order $o(h)$. Since $u_W(t+h)$ converges to $u_W(t)$ as
$h \to 0$, we get that
\begin{equation*}
\int_{-\infty}^{u_W(t)} \lambda_t(v) \, dv
\;=\; - \lim_{h\to 0}  \frac 1h \int_{u_W(t)}^{u_W(t+h)} 
\rho_t(v) \, dv\;.
\end{equation*}
We assume that $h>0$ and compute the right derivative of $u_W(t)$.
The left derivative is left to the reader.  Assume that
$\int_{(-\infty, u_W(t)]} \lambda_t(v) dv >0$ so that $u_W(t+h) <
u_W(t)$ for $h$ sufficiently small.  Since $\rho_t$ is a c\`adl\`ag
function and $u_W(t+h) < u_W(t)$ for $h$ sufficiently small, it
follows from the previous indentity that
\begin{equation*}
\lim_{h\downarrow 0}  \frac {u_W(t+h) - u_W(t)}h
\;=\; - \frac 1{\rho_t(u_W(t)-)}
\int_{-\infty}^{u_W(t)} \lambda_t(v) \, dv \;.
\end{equation*}
Similar identities can be obtained if $\int_{(-\infty, u_W(t)]}
\lambda_t(v) dv $ vanishes or is less than $0$ and for $h\uparrow 0$.
Thus, $u_W(\cdot)$ is differentiable both from the right and from the
left. To prove the lemma it remains to show that
\begin{equation*}
\frac {d \rho_t}{dW} (u_W(t)) \;=\;  \int_{-\infty}^{u_W(t)}
\lambda_t(v) \, dv \;.
\end{equation*}

Since $\lambda_t = \mf L_W \rho_t$, we have that
\begin{equation*}
\rho_t (u) \;=\; a_t \;+\; b_t W(u) \;+\; \int_0^u W(dv)\,
\int_0^v \lambda_t (w) \, dw
\end{equation*}
for some finite constants $a_t$, $b_t$. In particular, for any $u<v$,
\begin{equation*}
\frac {d \rho_t}{dW} (v) \;=\; \frac {d \rho_t}{dW} (u)
\;+\; \int_u^v \lambda_t (w) \, dw\;.
\end{equation*}
Since $\lambda_t$ belongs to $L^1(\bb R)$, letting $u\downarrow
-\infty$, we find that $d \rho_t/dW (u)$ converges to some constant
$c_t$ and that
\begin{equation*}
\frac {d \rho_t}{dW} (v) \;=\; c_t
\;+\; \int_{-\infty}^v \lambda_t (w) \, dw\;.
\end{equation*}
Take $u>0$.  Since $W(0) = 0$, integrating this identity with respect
to $dW$ in the interval $(0,u]$ we get that
\begin{equation*}
\rho_t (u) \;-\; \rho_t (0) \;=\; c_t W(u) \;+\;
\int_{(0,u]} W(dv) \int_{-\infty}^v \lambda_t(w) \, dw\;.
\end{equation*}
Dividing by $W(u)$, since $\rho_t$ is uniformly bounded and since
$\lim_{u\to\infty} W(u) = \infty$
\begin{equation*}
c_t \;=\; \lim_{u\to\infty}\frac {-1}{W(u)}
\int_{(0,u]} W(dv) \int_{-\infty}^v \lambda_t(w) \, dw \;.
\end{equation*}
Since the $L^1(\bb R)$ norm of $\rho_t$ is constant in time,
$\int_{-\infty}^\infty \lambda_t(w) \, dw =0$ for all $t>0$.  In
particular, the previous expression vanishes.  This concludes the
proof of the lemma.
\end{proof}

\medskip
\noindent{\bf Acknowledgments.} The authors wish to thank the Centro di
Ricerca Matematica Ennio De Giorgi of the Scuola Normale Superiore,
Pisa, Italy, where part of this work was done.

\end{document}